\documentclass[oneside]{article}
\usepackage{fancyhdr}
\title
{\LARGE{\bfseries{On degenerate coupled transport processes in porous media with memory phenomena}}}
\date{}
\author{\large{
Michal Bene\v s\footnote{Department of Mathematics,
Faculty of Civil Engineering, Czech Technical University in Prague,
Th\'{a}kurova 7, 166 29 Prague 6, Czech Republic,
E-mail: michal.benes@cvut.cz}\;
and
Igor Pa\v{z}anin\footnote{Department of Mathematics, Faculty of Science, University of Zagreb,
Bijeni\v{c}ka 30, 10000 Zagreb, Croatia,
E-mail: pazanin@math.hr}
}}

\usepackage{amssymb}
\usepackage{amsfonts}
\usepackage{graphicx}
\usepackage{amsmath}
\usepackage{latexsym}
\usepackage{amssymb}
\usepackage{pstricks}
\usepackage{amsmath}
\usepackage{graphics}
\usepackage{mathrsfs}
\usepackage{fancyhdr}
\usepackage{times}

 \newtheorem{thm}{Theorem}[section]
 
 \newtheorem{lem}[thm]{Lemma}

 \newtheorem{defn}[thm]{Definition}
 \newtheorem{rem}[thm]{Remark}
\numberwithin{equation}{section}

\newenvironment{proof}[1][Proof.]{\begin{trivlist}
\item[\hskip \labelsep {\bfseries #1}]}{\end{trivlist}}

\newcommand{\pom}{\partial\Omega}

\newcommand{\bfn}{\mbox{\boldmath{$n$}}}

\newcommand{\bfu}{\mbox{\boldmath{$u$}}}

\newcounter{constants}
\begin{document}

\maketitle

%
%
%
%
%
%
%
%
%

%



\begin{abstract}
In this paper we prove the existence of weak solutions
to degenerate parabolic systems
arising from the fully coupled moisture movement,
solute transport of dissolved species
and heat transfer
through porous materials.
Physically relevant mixed Dirichlet-Neumann boundary conditions
and initial conditions are considered.
Existence of a global weak solution
of the problem is proved by means of semidiscretization in time,
proving necessary uniform estimates
and by passing to the limit from discrete approximations.
Degeneration
occurs in the nonlinear transport coefficients
which are not assumed to be bounded below and above by positive
constants.
Degeneracies in transport coefficients
are overcome by proving suitable a-priori $L^{\infty}$-estimates based on
De Giorgi and Moser iteration technique.
\end{abstract}


\section{Introduction}
Let $\Omega$ be a bounded domain in $\mathbb{R}^2$, $\Omega \in
C^{0,1}$  and let $\Gamma_D$ and $\Gamma_N$ be open disjoint subsets of
$\pom$ (not necessarily connected) such that $\Gamma_D\neq\emptyset$
and the $\partial\Omega \backslash (\Gamma_D\cup\Gamma_N)$ is a
finite set. Let $T\in(0,\infty)$ be fixed throughout the paper, $I=(0,T)$
and $Q_T = \Omega\times I$ denotes the space-time cylinder,
$\Gamma_{DT} = \Gamma_D\times I$
and
$\Gamma_{NT}=\Gamma_N\times I$.
We shall study the following problem in $Q_{T}$:
\begin{align}
\partial_t [ \phi({x},r){S}(p) ]
&
=
\nabla\cdot[a({x},{p},{\vartheta},{r})\nabla p]
+
\alpha_1 f({x},p,c,\vartheta,r),
\label{strong:eq1a}
\\
\partial_t [\phi({x},r){S}(p) c]
&
=
\nabla\cdot[\phi({x},r){S}(p)D_w(x,p)\nabla c]
\nonumber
\\
&
\quad
+
\nabla\cdot[ c a({x},{p},{\vartheta},{r})\nabla { p}],
\label{strong:eq1b}
\\
\partial_t
\left[
\phi({x},r){S}(p) \vartheta + {\varrho}(x,r) \vartheta
\right]
&
=
\nabla \cdot [ \lambda({x},{p},{\vartheta},{r}) \nabla \vartheta ]
\nonumber
\\
&
\quad
+
\nabla\cdot
[\vartheta a({x},{p},{\vartheta},{r})\nabla p ]
+
\alpha_2 f({x},p,c,\vartheta,r),
\label{strong:eq1c}
\end{align}
where
\begin{equation}\label{def:transport_coef_a}
a({x},p,\vartheta,r) = \frac{k(x,r)k_{R}({S}(p))}{\mu(\vartheta)}.
\end{equation}

The system introduced above is coupled
with an integral condition
\begin{equation}\label{eq:memory_00}
r(x,t)
=
\int_0^t
f({x},p(x,s),c(x,s),\vartheta(x,s),r(x,s))
\,
{\rm d}s
\end{equation}
and completed by the mixed-type boundary conditions
\begin{align}
{p} &= 0,
\quad
{c} = 0,
\quad
\vartheta = 0
&&
{\rm on} \; \Gamma_{DT},
\label{strong:eq1f}
\\
\nabla {p}  \cdot \bfn
& = 0,
\quad
\nabla {c} \cdot \bfn
=
0,
\quad
\nabla \vartheta \cdot \bfn
=
0
&&
{\rm on} \;  \Gamma_{NT}
\label{strong:eq1i}
\end{align}
and the initial conditions
\begin{equation}
{p}(\cdot,0)  =  {p}_0,
\;
{c}(\cdot,0)  = {c}_0,
\;
\vartheta(\cdot,0)  = \vartheta_0
\qquad
{\rm in} \;  \Omega.
\label{strong:eq1l}
\end{equation}
The goal of this paper is to study the existence
of the so called weak solution to the degenerate fully coupled nonlinear system  \eqref{strong:eq1a}--\eqref{strong:eq1l}.
The problem under consideration covers a large range of problems including memory phenomena.
Namely, the system \eqref{strong:eq1a}--\eqref{strong:eq1l} arises from the coupled moisture movement,
transport of dissolved species and heat transfer
through the porous system \cite{Bear,PinderGray}.
Equations \eqref{strong:eq1a} and \eqref{strong:eq1b}
express the mass balance of water and dissolved species, respectively, in porous media
and  \eqref{strong:eq1c} represents the balance of heat energy in the porous system.
For simplicity, the gravity terms are not included since they do not affect the analysis.
For specific civil engineering applications,
we refer the reader to e.g. \cite{Gawin2006a,Ozbolt}.
Our problem has been motivated by doubly nonlinear systems appearing in
modelling of chemical reactions, heat transport and mass transfer in early age concrete
\cite{IshidaMaekawaKishi2007,Maekawa3,Maekawa4,Song2201}.

In \eqref{strong:eq1a}--\eqref{strong:eq1l},
${p} : Q_{T} \rightarrow \mathbb{R}$,
${c} : Q_{T} \rightarrow \mathbb{R}$,
$\vartheta : Q_{T} \rightarrow \mathbb{R}$
and
${r} : Q_{T} \rightarrow \mathbb{R}$
are the unknown functions.
In particular, $p$ corresponds to the water pressure,
$c$ represents concentration of dissolved species
and $\vartheta$ represents the temperature of the complete porous system.
Equations \eqref{strong:eq1a}, \eqref{strong:eq1c} and \eqref{eq:memory_00}
are encountered e.g. in the so called problem of ``hydratational heat''
when inner moisture sinks and heat sources are of special types.
In particular, the intensity of heat sources depends on the amount
of heat already developed, $f$ in \eqref{strong:eq1a} and \eqref{strong:eq1c}
depends on the unknown function
$r$ (the so called ``hydration degree'') given by \eqref{eq:memory_00}.
Further, $a: \Omega \times \mathbb{R}^3 \rightarrow \mathbb{R}$
represents the transport coefficient of the capillary water
given by \eqref{def:transport_coef_a}, where
$k$ is the intrinsic permeability,
$k_{R}$ denotes the relative permeability of the liquid water
and $\mu$ is the dynamic viscosity of the liquid water.
$D_w:\Omega \times\mathbb{R}\rightarrow \mathbb{R}$
is the capillary water diffusion coefficient,
$S:\mathbb{R}\rightarrow \mathbb{R}$
represents degree of saturation of the pores with liquid water,
$\phi:\Omega \times\mathbb{R}\rightarrow \mathbb{R}$ is porosity,
$\varrho:\Omega \times\mathbb{R}\rightarrow \mathbb{R}$
is the density of solid skeleton in the porous system.
Further,
$\lambda:\Omega \times\mathbb{R}^3\rightarrow \mathbb{R}$ is the thermal conductivity
of porous material. $\alpha_1$ and $\alpha_2$ are material constants.
Note that the density of water is assumed to be constant in the model and normalized to one.
$\bfn$ is the outward unit normal vector with respect to the boundary of $\Omega$.
Finally,
${p}_{0} : \Omega \rightarrow \mathbb{R}$,
${c}_{0} : \Omega \rightarrow \mathbb{R}$
and
$\vartheta_{0} : \Omega\rightarrow \mathbb{R}$
are given functions describing initial state of the system.

Typical forms of $S$ and $k_{R}$ can be found e.g. in
\cite{Genuchten1980,gerkegenuchten,gerkegenuchten_2,PinderGray}
with applications to water movement in soils or structured rock masses
or in \cite{Gawin2006a,Gawin2011a,Gawin2011b}
concerning transport processes in concrete.
It follows that positive functions $S(\cdot)$ and $k_{R}(S(\cdot))$ are typically increasing on $(-\infty,0)$ and
$S'(\cdot)$ and $k_{R}(S(\cdot))$ tend to zero as
$p \rightarrow -\infty$. Hence, \eqref{strong:eq1a} and \eqref{strong:eq1b}
are degenerate parabolic equations
where the degeneracy occurs in both elliptic as well as parabolic parts.
The degeneracy in the elliptic part of \eqref{strong:eq1a}
can be transformed only to the parabolic term
using the so called Kirchhoff transformation
\begin{displaymath}%
v := \kappa(p) = \int\limits_{0}^{p} k_{R}({S}(s))
{\rm d}s.
\end{displaymath}
The existence of the weak solution $v$ for the resulting transformed problem
follows from Alt and Luckhaus \cite{AltLuckhaus1983}.
However, due to the degeneracy of the problem
($k_{R}$ is not assumed to be bounded below by a positive constant)
we are not able to ensure that $p=\kappa^{-1}(v)$ solves the original problem.
Therefore we treat directly the doubly degenerate problem  \eqref{strong:eq1a}
and omit degeneracies proving $L^{\infty}$-estimates for the solutions of the approximate
problems.

\paragraph{A brief bibliographical survey.}
Nowadays, description of heat, moisture or soluble/non-soluble
contaminant transport in concrete, soil or rock porous matrix is
frequently based on time dependent models.
Coupled transport processes (diffusion processes, heat conduction, moister flow,
contaminant transport or coupled flows through porous media) are typically associated with
systems of strongly nonlinear degenerate parabolic partial differential
equations of type (written in terms of operators
${A}$, ${B}$, ${F}$)
\begin{equation}\label{eq:par1}
 \partial_t {B}(\bfu) -\nabla\cdot
{A}(\bfu,\nabla \bfu)
=
{F}(\bfu),
\end{equation}
where $\bfu$ stands for the unknown vector of state variables.
There is no complete theory for
such general problems. However,
some particular results assuming special structure of
operators ${A}$ and ${B}$ and growth
conditions on ${F}$ can be found in
the literature.
Most theoretical results on parabolic systems exclude the case of non-symmetrical parabolic parts \cite{AltLuckhaus1983,FiloKacur1995,Kacur1990a}.
Giaquinta and Modica in~\cite{GiaquintaModica1987} proved the local-in-time solvability of quasilinear
diagonal parabolic systems with nonlinear boundary conditions (without assuming any growth condition), see also \cite{Weidemaier1991}.
The existence of weak solutions to more general non-diagonal systems like \eqref{eq:par1} subject to mixed
boundary conditions has been proven in~\cite{AltLuckhaus1983}. The authors
proved an existence result assuming the operator ${B}$
to be only (weak) monotone and subgradient.
This
result has been extended in~\cite{FiloKacur1995}, where the authors presented
the local existence of the weak
solutions for the system with nonlinear Neumann boundary conditions
and under more general growth conditions on nonlinearities in $\bfu$.
These results, however, are not applicable if ${B}$ does not take the
subgradient structure, which is typical of coupled transport models in porous media.
Thus, the analysis needs to exploit the specific structure of such problems.
The existence
of a local-in-time strong solution
for moisture and heat transfer in multi-layer porous structures
modelling by the doubly nonlinear parabolic system
is provided in \cite{BenesZeman}.
In \cite{Vala2002}, the author
proved the existence of the solution to the purely diffusive
hygro-thermal model allowing non-symmetrical operators ${B}$, but requiring
non-realistic symmetry in the elliptic part.
In
\cite{degond,jungel2000}, the authors studied the existence, uniqueness and regularity
of coupled quasilinear equations modeling evolution of fluid species influenced by thermal, electrical
and diffusive forces.
In \cite{LiSun2010,LiSunWang2010,LiSun2012}, the authors studied a model of specific structure
of a heat and mass transfer arising from textile industry and
proved the global existence for one-dimensional problems in \cite{LiSun2010,LiSunWang2010}
and three-dimensional problems in \cite{LiSun2012}.
In \cite{harris2017}, the authors proved
the existence of the weak solutions to systems modeling the consolidation of saturated porous media.

In \cite{Pluschke1992,Pluschke1997}, the author proved the local existence of weak solutions
to degenerate quasilinear problems, where the coefficient function in front of the time derivative
may vanish at a set of zero measure. The main result is proved by means
of semi-discretization in time and proving $L^{\infty}$-estimates for approximates
in order to omit a growth limitations in nonlinearities and the right hand side.

In \cite{Ning1993}, the author
studied an initial boundary value problem for the nonlinear degenerate
parabolic equation of type \eqref{eq:par1} with evolutionary boundary conditions.
Existence and uniqueness
were established through some discrete schemes combined with parabolic regularization
and error estimates for these schemes were presented.
In a slightly different form, taking $\phi = 1$ and $f=0$
and assuming different degeneration features,
problem \eqref{strong:eq1a}--\eqref{strong:eq1b}
was studied in \cite{Ning1990c,Ning1990,Ning1990b}.

From the numerical point of view,
scalar degenerate problems similar to \eqref{strong:eq1a}
were treated in  \cite{Kacur1999,Kacur2001}. The author proposed
a nonstandard  approximation scheme based on the
relaxation method in order to control the degeneracy in the problem.

A profound investigation of problems with integral conditions
connected with equations \eqref{strong:eq1a} and \eqref{eq:memory_00}
can be found in \cite{Doktor1985} and \cite{Rektorys}
in case of nondegenerate linear elliptic and parabolic parts
and in \cite{Kacur1999} assuming strongly nonlinear and degenerate
scalar parabolic problems.

\paragraph{Outline of the paper.}
In the present paper we extend our previous existence result for coupled heat and
mass flows in porous media \cite{BenesKrupicka2016}
to more general degenerate problem
modeling coupled moisture, solute and heat transport in porous media
including memory phenomena.
This leads to a fully nonlinear degenerate parabolic system
coupled with an integral condition
and including natural (critical) growths
and with degeneracies in transport coefficients.
The rest of this paper is organized as follows.  In Section~\ref{sec:prelim},
we briefly introduce basic notation and suitable function spaces
and specify our assumptions on data
and coefficient functions in the problem.
In Section~\ref{sec:main_result}, we formulate the problem in the
variational sense and state the main result, the global-in-time
existence of the weak solution to the problem \eqref{strong:eq1a}--\eqref{strong:eq1l}.
The main result is proved by an approximation procedure in Section~\ref{sec:proof_main}.
First, we formulate the semi-discrete scheme and prove
the existence of the solution to the corresponding recurrence steady problem.
The crucial a-priori estimates and uniform boundness
of time discrete approximations are proved in Section~\ref{sec:estimates}.
Finally, we conclude that
solutions of the semi-discrete scheme converge and that the limit
corresponds to the solution of the original problem.

\begin{rem}
The present analysis can be straightforwardly extended to a setting with
nonhomogeneous boundary conditions \eqref{strong:eq1f}
(see \cite{BenesKrupicka2016} for details or
\cite[Paragraph~1.10 on page 324]{AltLuckhaus1983}).
Here we work with homogeneous boundary conditions to
avoid unnecessary technicalities in the existence result.
\end{rem}

\section{Preliminaries}
\label{sec:prelim}
\subsection{Notations and some function spaces}
\label{notations}
Vectors and vector functions are denoted by boldface letters.
Throughout the paper, we will always use positive constants ${C}$,
${C}_1$, ${C}_2$, $\dots$, which are not specified and which may
differ from line to line.
In what follows, we suppose
$s,q,s'\in [1,\infty]$, $s'$ denotes the conjugate exponent to $s>1$,
${1}/{s} + {1}/{s'} = 1$.
$L^s(\Omega)$ represents the usual
Lebesgue space equipped with the norm $\|\cdot\|_{L^s(\Omega)}$ and
$W^{k,s}(\Omega)$, $k\geq 0$ ($k$ need not to be an integer, see
\cite{KufFucJoh1977}), denotes the usual
Sobolev-Slobodecki space with the norm $\|\cdot\|_{W^{k,s}(\Omega)}$.
We define
$
W^{1,2}_{\Gamma_D}(\Omega)
:=
\left\{
v \in W^{1,2}(\Omega); \,   v \big|_{\Gamma_D} = 0
\right\}$.
By  $E^*$ we denote the space of all continuous, linear forms on Banach space $E$
and by $\langle \cdot,\cdot \rangle$ we denote the
duality between $E$ and $E^*$.
By $L^s(I;E)$ we denote the Bochner space (see \cite{AdamsFournier1992}).
Therefore, $L^s(I;E)^*=L^{s'}(I;E^*)$.

\subsection{Structure and data properties}
\label{Structure and data properties}

We start by introducing several structural assumptions on functions in
\eqref{strong:eq1a}--\eqref{strong:eq1l}:
\begin{itemize}

\item[(i)]
$S \in C^{1}(\mathbb{R})$ is a positive and strictly monotone function such that
\begin{align}
& 0 < S(\xi) \leq S_{s} < +\infty
&&
\forall \xi \in \mathbb{R} \quad (S_{s} = {\rm const}),
\label{con11a}
\\
&  \left(S(\xi_1)-S(\xi_2)\right)(\xi_1 - \xi_2) > 0
&&
\forall \xi_1,\xi_2 \in \mathbb{R}, \;  \xi_1 \neq \xi_2.
\label{con11b}
\end{align}

\item[(ii)]
$a$, $a_1$, $k$, $k_{R}$, $\mu$, $\varrho$, $D_w$, $\lambda$ are continuous functions, $a_1$ strictly increasing,
satisfying
\begin{align}
&
0 < k_1 \leq k(x,\xi) \leq k_2 < +\infty \quad (k_1,k_2 = {\rm const})
&&
\forall \xi \in \mathbb{R},
\;
{x} \in \Omega,
\label{cond:intr_perm}
\\
&
k_{R} \in C([0,S_s]), \;
\left(k_{R}(\xi_1)-k_{R}(\xi_2)\right)(\xi_1 - \xi_2) > 0
&&
\forall \xi_1,\xi_2 \in [0,S_s], \;
\label{cond:rel_perm_I}
\\
&
&&
\xi_1 \neq \xi_2,
\nonumber
\\
&
0 <  k_{R}(\xi)
&&
\forall \xi \in [0,S_s],
\label{cond:rel_perm_II}
\\
&
0 <  \mu_1 \leq  \mu(\xi) \leq  \mu_2 < +\infty \quad ( \mu_1, \mu_2 = {\rm const})
&&
\forall \xi \in \mathbb{R},
\label{cond:viscosity}
\\
&
0 < \varrho_1 \leq  \varrho(x,\xi) \leq  \varrho_2 < +\infty
\quad
( \varrho_1, \varrho_2 = {\rm const})
&&
\forall \xi \in \mathbb{R},
\;
x \in \Omega,
\label{cond:density}
\\
&
0 < a_1(\xi_1) \leq a({x},\xi_1,\xi_2,\xi_3) \leq a_2 < +\infty
&&
\forall \xi_1,\xi_2,\xi_3 \in \mathbb{R},
\label{con12a}
\\
&
&&
x \in \Omega \quad (a_2 =  {\rm const}),
\nonumber
\\
&
0 <  D_w({x},\xi)
&&
\forall \xi \in \mathbb{R},
\;
x \in \Omega,
\label{con12b}
\\
&
0 < \lambda({x},\xi_1,\xi_2,\xi_3)
&&
\forall \xi_1,\xi_2,\xi_3 \in \mathbb{R},
\label{con12c}
\\
&
&&
\;
x \in \Omega.
\nonumber
\end{align}

\item[(iii)] The function $\phi$ is Lipschitz continuous with respect
to the second variable, i.e.
there exists a constant $C_{\phi}>0$ such that
\begin{align}\label{lipschitz_phi}
| \phi(x,\xi_1)-\phi(x,\xi_2) |
&
\leq
C_{\phi}
|\xi_1 - \xi_2|
&&
\forall \xi_1,\xi_2 \in \mathbb{R},
\;
x \in \Omega
\end{align}
and
\begin{equation}\label{bound_phi}
0 < \phi_1 \leq \phi(x,\xi) \leq \phi_2 < +\infty
\quad
\forall \xi \in \mathbb{R},
\;
{x} \in \Omega
\quad
(\phi_1,\phi_2 = {\rm const}).
\end{equation}

\item[(iv)]
$f$ is Lipschitz continuous in all respective variables and
there exists an increasing positive bounded function $\tilde{f}$ such that
$(C_f = {\rm const})$
\begin{equation}\label{assum:bound_f}
|f({x},\xi_1, \xi_2, \xi_3, \xi_4)|
\leq
\tilde{f}({\xi_1})
\leq
C_f
\quad
\forall \xi_1, \xi_2, \xi_3, \xi_4 \in \mathbb{R},
\;
x \in \Omega.
\end{equation}

\item[(v)]
We assume that there exists a non-increasing positive function $M$ such that
\begin{equation}\label{def:M}
M(\xi) \leq \frac{a_1(\xi)}{S'(\xi)} \qquad \forall\xi \in \mathbb{R}
\end{equation}

and
\begin{equation}
\lim_{\delta\rightarrow 0_{+}}
\frac{\tilde{f}({S^{-1} (\delta)})}{M({S^{-1} (\delta)})\delta} = 0.
\end{equation}

\item[(vi)] (Initial data)
Assume ${c}_{0}$, ${\vartheta}_{0} \in W^{1,2}\cap L^{\infty}(\Omega)$
and
${p}_{0} \in L^{\infty}(\Omega)$
such that
\begin{equation}\label{cond:lower_bound_pressure}
-\infty < {p}_{1} < {p}_0(\cdot) \leq 0   \qquad \textmd{ a.e. in } \Omega \quad
({p}_{1}={\rm const}).
\end{equation}

\end{itemize}
Throughout the paper the hypotheses (i)--(vi) will be assumed.

\section{The main result}\label{sec:main_result}
The aim of this paper is to prove the existence of a weak solution to the problem
\eqref{strong:eq1a}--\eqref{strong:eq1l}.
We first reformulate the problem in a variational sense.
\begin{defn}
\label{def_weak_solution}
A weak solution of \eqref{strong:eq1a}--\eqref{strong:eq1l}
is a foursome $[p,c,\vartheta,r]$ such that
\begin{align*}
& p \in L^2(I;W_{\Gamma_D}^{1,2}(\Omega)) ,
\\
& c \in  L^2(I;W_{\Gamma_D}^{1,2}(\Omega)) \cap L^{\infty}({Q_T}),
\\
& \vartheta \in  L^2(I;W_{\Gamma_D}^{1,2}(\Omega)) \cap L^{\infty}({Q_T}),
\\
&   r \in {C([0,T];L^{\infty}(\Omega))}
\end{align*}
and
\begin{multline}
\label{weak_form_01}
-\int_{Q_T}
\phi({x},r){S}(p) \partial_t\zeta
{\,{\rm d}x} {\rm d}t
+
\int_{Q_T}
{a({x},p,\vartheta,r){\nabla p}} \cdot\nabla\zeta
{\,{\rm d}x} {\rm d}t
\\
=
\int_{\Omega}   \phi({x},r_{0}){S}(p_{0}) \zeta(x,0) {\,{\rm d}x}
+
\int_{Q_T}
\alpha_1 f({x},p,c,\vartheta,r)
\zeta
{\,{\rm d}x} {\rm d}t
\end{multline}
for any  $\zeta \in {L^2(I;W^{1,2}_{\Gamma_D}(\Omega))\cap W^{1,1}(I;L^{1}(\Omega))}$
with $\zeta(\cdot,T)=0$;

\begin{multline}
\label{weak_form_02}
-\int_{Q_T}  \phi({x},r){S}(p) c \, \partial_t\eta
{\,{\rm d}x} {\rm d}t
+
\int_{Q_T}
\phi({x},r){S}(p)D_w(x,p)\nabla c \cdot\nabla \eta
{\,{\rm d}x} {\rm d}t
\\
+
\int_{Q_T}
 c \, a({x},{p},{\vartheta},{r})\nabla { p} \cdot\nabla\eta
{\,{\rm d}x} {\rm d}t
=
\int_{\Omega} \phi({x},r_{0}){S}(p_{0}) {c}_0 \, \eta(x,0) {\,{\rm d}x}
\end{multline}
for any  $\eta \in L^2(I;W^{1,2}_{\Gamma_D}(\Omega))\cap W^{1,1}(I;L^{1}(\Omega))$
with $\eta(\cdot,T)=0$;

\begin{align}
\label{weak_form_03}
&
\quad
-\int_{Q_T}
\left[
\phi({x},r){S}(p) + {\varrho}(x,r)
\right]
\vartheta \,
\partial_t\psi
{\,{\rm d}x} {\rm d}t
\nonumber
\\
&
\quad
\int_{Q_T}
\lambda({x},{p},{\vartheta},{r}) \nabla\vartheta \cdot \nabla\psi
{\,{\rm d}x} {\rm d}t
+
\int_{Q_T}
\vartheta {a({x},{p},{\vartheta},{r}) {\nabla p}}
\cdot \nabla\psi
{\,{\rm d}x} {\rm d}t
\nonumber
\\
&
=
\int_{\Omega}
[  \phi({x},r_{0}){S}(p_{0}) + {\varrho}(x,{r}_0)]  \vartheta_{0}  \psi(x,0) {\,{\rm d}x}
+
\int_{Q_T}
\alpha_2 f({x},p,c,\vartheta,r)
\psi
{\,{\rm d}x} {\rm d}t
\end{align}
for any  $\psi \in L^2(I;W^{1,2}_{\Gamma_D}(\Omega))\cap W^{1,1}(I;L^{1}(\Omega))$
with $\psi(\cdot,T)=0$,
where
\begin{equation}
\label{eq:memory_weak_form}
r(t)
=
\int_0^t
f({x},p(x,s),c(x,s),\vartheta(x,s),r(x,s))
\,
{\rm d}s
\quad
\textmd{ in }  {L^{\infty}(\Omega)}
\textmd{ for all } t \in [0,T].
\end{equation}
\end{defn}

The main result of this paper reads as follows:
\begin{thm}[Main result]\label{main_result}
Let the assumptions {\rm (i)--(vi)} be satisfied.
Then there exists at
least one weak solution of the
system  \eqref{strong:eq1a}--\eqref{strong:eq1l}.
\end{thm}

To prove the main result of the paper we use the method
of semidiscretization in time by constructing temporal approximations
and limiting procedure.
The proof can be divided into three steps.
In the first step, we approximate our problem by means of a semi-implicit time discretization
scheme (which preserve the pseudo-monotone structure of the discrete problem)
and prove the existence and $W^{1,s}(\Omega)$-regularity (with some $s>2$) of discrete approximations.
In the second step we derive necessary a-priori estimates.
The key point is to establish $L^{\infty}$-estimates to overcome degeneracies in transport coefficients.
Finally, in the third step we construct temporal interpolants and
pass to the limit from discrete approximations.

\section{Proof of the main result}\label{sec:proof_main}
\subsection{Approximations}\label{sec:approximations}

Applying the method of discretization in time, we divide the interval
$[0,T]$ into $n$ subintervals of lengths  ${h}:= T/n$ (a time step),
replace the time derivatives
by the corresponding difference quotients
and the integral in \eqref{eq:memory_weak_form} by a sum.
In this way,
we approximate the problem \eqref{strong:eq1a}--\eqref{strong:eq1l}
by a semi-implicit time discretization scheme
and re-formulate the problem in a weak sense.

Let us consider
$p^0_{n} := p_{0}$,
$c^0_{n} := c_{0}$,
$\vartheta^{0}_{n} := \vartheta_{0}$
and
$r^0_{n} := {0}$ a.e. on $\Omega$.
We now define,
in each time step $i=1,\dots,n$,
a foursome
$[p^{i}_{n},c^{i}_{n},\vartheta^{i}_{n},r^{i}_{n}]$
as a solution of the
following recurrence steady problem:
for a given foursome
$[p^{i-1}_{n},c^{i-1}_{n},\vartheta^{i-1}_{n},r^{i-1}_{n}]$,
$i=1,2,\dots,n$,
${p}^{i-1}_{n} \in  L^{\infty}(\Omega)$,
${c}^{i-1}_{n} \in W^{1,2}(\Omega) \cap L^{\infty}(\Omega)$,
$\vartheta^{i-1}_{n} \in  W^{1,2}(\Omega) \cap L^{\infty}(\Omega)$
and
${r}^{i-1}_{n} \in  W^{1,2}(\Omega) \cap L^{\infty}(\Omega)$,
find  $[p^{i}_{n},c^{i}_{n},\vartheta^{i}_{n},r^{i}_{n}]$,
such that
${p}^{i}_{n} \in  W_{\Gamma_D}^{1,2}(\Omega) \cap L^{\infty}(\Omega)$,
${c}^{i}_{n} \in  W_{\Gamma_D}^{1,2}(\Omega) \cap L^{\infty}(\Omega)$,
$\vartheta^{i}_{n} \in  W_{\Gamma_D}^{1,2}(\Omega) \cap L^{\infty}(\Omega)$,
${r}^{i}_{n} \in  W^{1,2}(\Omega) \cap L^{\infty}(\Omega)$
and
\begin{align}
\label{approximate_problem_01}
&
\int_{\Omega}
\frac{ \phi({x},{r}^{i}_{n}){S}({p}^{i}_{n}) - \phi({x},{r}^{i-1}_{n}){S}({p}^{i-1}_{n}) }{h} \zeta
{\,{\rm d}x}
\nonumber
\\
&
+
\int_{\Omega}
a({x},{{p}_{n}^{i}},\vartheta_{n}^{i-1},{r}_{n}^{i-1})
\nabla {p}_{n}^{i}
\cdot\nabla\zeta
{\,{\rm d}x}
\nonumber
\\
=
&
\int_{\Omega}
\alpha_1 f({x},{{p}_{n}^{i}},{c}_{n}^{i-1},\vartheta_{n}^{i-1},{r}_{n}^{i-1})
\zeta
{\,{\rm d}x}
\end{align}
for any $\zeta \in {W_{\Gamma_D}^{1,2}(\Omega)}$;

\begin{align}\label{approximate_problem_02}
&
\quad
\int_{\Omega}
\frac{ \phi({x},{r}^{i}_{n}){S}({p}^{i}_{n}){c}^{i}_{n}
-
\phi({x},{r}^{i-1}_{n}){S}({p}^{i-1}_{n}){c}^{i-1}_{n} }{h} \eta
{\,{\rm d}x}
\nonumber
\\
&
\quad
+
\int_{\Omega}
\phi({x},{r}^{i}_{n}){S}({p}^{i}_{n})  D_w(x,{p}_{n}^{i})
\nabla {c}^{i}_{n} \cdot \nabla\eta
{\,{\rm d}x}
\nonumber
\\
&
\quad
+
\int_{\Omega}
{c}^{i}_{n} a({x},{{p}_{n}^{i}},\vartheta_{n}^{i-1},{r}_{n}^{i-1})
\nabla {p}_{n}^{i}
\cdot \nabla \eta
{\,{\rm d}x}
\nonumber
\\
&
=
0
\end{align}
for any  $\eta \in {W_{\Gamma_D}^{1,2}(\Omega)}$;

\begin{align}
\label{approximate_problem_03}
&
\int_{\Omega}
\frac{
\phi({x},{r}^{i}_{n}){S}({p}^{i}_{n}) \vartheta^{i}_{n}
-
\phi({x},{r}^{i-1}_{n}){S}({p}^{i-1}_{n}) \vartheta_{n}^{i-1} }{h}
\psi
{\,{\rm d}x}
\nonumber
\\
&
+
\int_{\Omega}
\frac{ \varrho(x,r_{n}^{i})\vartheta_{n}^{i} - \varrho(x,r_{n}^{i-1})\vartheta_{n}^{i-1} }{h}  \psi
{\,{\rm d}x}
\nonumber
\\
&
+
\int_{\Omega}
\lambda({x},{{p}_{n}^{i-1}},\vartheta_{n}^{i-1},{r}_{n}^{i-1}) \nabla \vartheta_{n}^{i} \cdot \nabla\psi
{\,{\rm d}x}
\nonumber
\\
&
+
\int_{\Omega}
\vartheta_{n}^{i}
a({x},{{p}_{n}^{i}},\vartheta_{n}^{i-1},{r}_{n}^{i-1})
\nabla {p}_{n}^{i}
\cdot \nabla\psi
{\,{\rm d}x}
\nonumber
\\
=
&
\int_{\Omega}
\alpha_2 f({x},{{p}_{n}^{i}},{c}_{n}^{i-1},\vartheta_{n}^{i-1},{r}_{n}^{i-1})
\psi
{\,{\rm d}x}
\end{align}
for any $\psi \in {W_{\Gamma_D}^{1,2}(\Omega)}$
and
\begin{align}
&
r_{n}^{i}
=
{h} \sum_{j=1}^{i}
f({x},{p}_{n}^{j},{c}_{n}^{j-1},\vartheta_{n}^{j-1},{r}_{n}^{j-1}),
\quad
i=1,\dots,n,
\label{approximate_problem_04a}
\\
&
r_{n}^{0}(x) = 0.
\label{approximate_problem_04b}
\end{align}

\begin{thm}[Existence of the solution to
\eqref{approximate_problem_01}--\eqref{approximate_problem_04a}]
\label{thm:aprox}
Let
${p}^{i-1}_{n} \in  L^{\infty}(\Omega)$,
${c}^{i-1}_{n} \in W^{1,2}(\Omega) \cap L^{\infty}(\Omega)$,
$\vartheta^{i-1}_{n} \in W^{1,2}(\Omega) \cap L^{\infty}(\Omega)$
and
${r}^{i-1}_{n} \in  W^{1,2}(\Omega) \cap L^{\infty}(\Omega)$
be given and the Assumptions {\rm (i)--(vi)} be satisfied.
Then there exists $[{p}^{i}_{n},{c}^{i}_{n},\vartheta^{i}_{n},{r}^{i}_{n}]$,
such that
${p}^{i}_{n} \in  W_{\Gamma_D}^{1,s}(\Omega)$,
${c}^{i}_{n} \in  W_{\Gamma_D}^{1,s}(\Omega)$,
$\vartheta^{i}_{n} \in  W_{\Gamma_D}^{1,s}(\Omega)$ with some $s>2$,
and ${r}^{i}_{n} \in  W^{1,2}(\Omega) \cap L^{\infty}(\Omega)$,
satisfying
\eqref{approximate_problem_01}--\eqref{approximate_problem_04a}.
\end{thm}
\begin{rem}
By Theorem~\ref{thm:aprox} and
the embeddings $W_{\Gamma_D}^{1,s}(\Omega) \hookrightarrow L^{\infty}(\Omega)$
(recall that $s>2$ and $\Omega \subset \mathbb{R}^2$)
and
$W_{\Gamma_D}^{1,s}(\Omega) \hookrightarrow W^{1,2}(\Omega)$
we are able to solve \eqref{approximate_problem_01}--\eqref{approximate_problem_04a}
recursively for
$[p^{i}_{n},c^{i}_{n},\vartheta^{i}_{n},r^{i}_{n}]$
by the already known $[p^{i-1}_{n},c^{i-1}_{n},\vartheta^{i-1}_{n},r^{i-1}_{n}]$,
such that we obtain
\begin{align*}
&   {p}^{i}_{n} \in  W_{\Gamma_D}^{1,2}(\Omega) \cap L^{\infty}(\Omega),
\\
&   {c}^{i}_{n} \in  W_{\Gamma_D}^{1,2}(\Omega) \cap L^{\infty}(\Omega),
\\
&   {\vartheta}^{i}_{n} \in  W_{\Gamma_D}^{1,2}(\Omega) \cap L^{\infty}(\Omega),
\\
&   {r}^{i}_{n} \in  W^{1,2}(\Omega) \cap L^{\infty}(\Omega)
\\
\end{align*}
for all $i=1,\dots,n$.
\end{rem}

Before proving Theorem~\ref{thm:aprox}, we present two auxiliary results,
formulated in
Theorem~\ref{thm:max_principle_pressure}
and
Lemma~\ref{lem:iteration}.

\begin{thm}[Weak maximum principle for pressure approximations]
\label{thm:max_principle_pressure}
Let
${p}^{i}_{n} \in  W_{\Gamma_D}^{1,s}(\Omega)$,
${c}^{i}_{n} \in  W_{\Gamma_D}^{1,s}(\Omega)$,
$\vartheta^{i}_{n} \in  W_{\Gamma_D}^{1,s}(\Omega)$ with some $s>2$,
and ${r}^{i}_{n} \in  W^{1,2}(\Omega) \cap L^{\infty}(\Omega)$
solve
\eqref{approximate_problem_01}--\eqref{approximate_problem_04b}
successively for $i=1,\dots,n$.
Then there exists $\ell$ (independent of $n$) such that
\begin{equation}\label{est:p_uniform_bound}
{p}_{n}^{i} \geq \ell \; \textmd{ almost everywhere in } \Omega
\textmd{ and for all } \; i=1,2,\dots,n.
\end{equation}
\end{thm}

To prove Theorem~\ref{thm:max_principle_pressure} we need the following lemma:
\begin{lem}[See e.g. Proposition~4.2 in \cite{Ning1990}]
\label{lem:iteration}
If a nonnegative sequence  $\left\{ Z_{j} \right\}$ satisfies
\begin{equation*}
Z_{j+1} \leq \gamma 4^{j} Z_{j}^{\tau+1} \qquad (\tau>0),
\end{equation*}
then
\begin{equation}\label{limit_00}
\lim_{j \rightarrow +\infty} Z_{j} = 0
\end{equation}
provided that
$$
\gamma \leq Z_{0}^{-\tau} 4^{-1/\tau}.
$$
\end{lem}
\begin{proof}
The proof follows from the proof of Lemma~4.1.1 in \cite{Wu2006}.
\end{proof}

\begin{proof}[Proof of Theorem~\ref{thm:max_principle_pressure}]
The proof is based on the De~Giorgi iteration technique,
see e.g. \cite[Chapter~4]{Wu2006} or \cite{Ladyzhenskaya1968}.
Let $k \in \mathbb{R}$ and set
\begin{equation*}
(\phi-k)_{-}
\equiv
\left\{
\begin{array}{ll}
\phi-k ,         &  \phi < k ,
\\
0 ,                   & \phi \geq k.
\end{array} \right.
\end{equation*}
For $k < {p}_1$ (here, ${p}_1$ is taken from \eqref{cond:lower_bound_pressure}) we have
$\zeta = (S({p}_{n}^{i})-S(k))_{-} \in W_{\Gamma_D}^{1,2}(\Omega)$ and
thus we may choose $\zeta = (S({p}_{n}^{i})-S(k))_{-}$
as a test function in
\eqref{approximate_problem_01}.
It is a matter of a simple technical computation
to arrive at the estimate
\begin{align}
\label{est:901}
&
\frac{1}{2h}
\int_{\Omega}
\phi(x,r_{n}^{i})
|(S({p}_{n}^{i})-S(k))_{-}|^2
{\,{\rm d}x}
\nonumber
\\
&
-
\frac{1}{2h}
\int_{\Omega}
\phi(x,r_{n}^{i-1})
|(S({p}_{n}^{i-1})-S(k))_{-}|^2
{\,{\rm d}x}
\nonumber
\\
&
+
\int_{\Omega}
a({x},{{p}_{n}^{i}},\vartheta_{n}^{i-1},{r}_{n}^{i-1})
\frac{1}{S'({p}_{n}^{i})}
|\nabla (S({p}_{n}^{i}) - S(k))_{-}|^2
{\,{\rm d}x}
\nonumber
\\
\leq
&
\int_{\Omega}
\alpha_1 f({x},{{p}_{n}^{i}}, {c}_{n}^{i-1}, \vartheta_{n}^{i-1}, {r}_{n}^{i-1})
(S({p}_{n}^{i})-S(k))_{-}
{\,{\rm d}x}
\nonumber
\\
&
-
\frac{1}{2h}
\int_{\Omega}
\left[
\phi(x,r_{n}^{i})-\phi(x,r_{n}^{i-1})
\right]
(S({p}_{n}^{i})+S(k))
(S({p}_{n}^{i})-S(k))_{-}
{\,{\rm d}x}.
\end{align}
Using the Lipschitz continuity of $\phi$ with respect to $r$, see \eqref{lipschitz_phi},
and using \eqref{approximate_problem_04a},
we can write
\begin{align}\label{est:902}
| \phi(x,r_{n}^{i})-\phi(x,r_{n}^{i-1}) |
&
\leq
C_{\phi}
|r_{n}^{i} - r_{n}^{i-1}|
\nonumber
\\
&
\leq
{h}
C_{\phi}
|f({x},{{p}_{n}^{i}}, {c}_{n}^{i-1}, \vartheta_{n}^{i-1}, {r}_{n}^{i-1})|.
\end{align}
Let us denote
\begin{equation}
I_{k}(i) := \int_{\Omega}
\phi(x,r_{n}^{i})
|(S({p}_{n}^{i})-S(k))_{-}|^2
{\,{\rm d}x},
\quad
i=1,\dots,n
\end{equation}
and let $I_{k}$ attains its maximum at $i=m$, i.e.
\begin{equation}\label{est:600a}
I_{k}(m) =  \max_{i=1,\dots,n} I_{k}(i)
\end{equation}
and, in other words,
\begin{equation}
I_{k}(m) \geq  I_{k}(i)  \quad   \textmd{ for all }   i=1,2,\dots,n.
\end{equation}
From this and in view of \eqref{est:901} and \eqref{est:902} we have
\begin{align}
\label{est:601}
&
\quad
\int_{\Omega}
a({x},{{p}_{n}^{m}},\vartheta_{n}^{m-1},{r}_{n}^{m-1})
\frac{1}{S'({p}_{n}^{m})}
|\nabla (S({p}_{n}^{m}) - S(k))_{-}|^2
{\,{\rm d}x}
\nonumber
\\
&
\leq
(C_{\phi}S_s + |\alpha_1|)
\int_{\Omega}
|f({x},{{p}_{n}^{m}}, {c}_{n}^{m-1}, \vartheta_{n}^{m-1}, {r}_{n}^{m-1})
(S({p}_{n}^{m}) - S(k))_{-}|
{\,{\rm d}x}.
\end{align}
Further, using \eqref{con12a} leads to
\begin{align}
\label{est:602}
&
\quad
\int_{\Omega}
\frac{a_1({{p}_{n}^{m}})}{S'({p}_{n}^{m})}
|\nabla (S({p}_{n}^{m}) - S(k))_{-}|^2
{\,{\rm d}x}
\nonumber
\\
&
\leq
(C_{\phi}S_s + |\alpha_1|)
\int_{\Omega}
|f({x},{{p}_{n}^{m}}, {c}_{n}^{m-1}, \vartheta_{n}^{m-1}, {r}_{n}^{m-1})
(S({p}_{n}^{m}) - S(k))_{-}|
{\,{\rm d}x}.
\end{align}
On the other hand, by Assumption (v), namely, the inequality \eqref{def:M},
we see that
\begin{equation}
\label{est:603}
M(k)
\int_{\Omega}
|\nabla (S({p}_{n}^{m}) - S(k))_{-}|^2
{\,{\rm d}x}
\leq
\int_{\Omega}
\frac{a({{p}_{n}^{m}})}{S'({p}_{n}^{m})}
|\nabla (S({p}_{n}^{m}) - S(k))_{-}|^2
{\,{\rm d}x}.
\end{equation}
Using the embedding theorem gives {(recall that $\Omega$ is a two-dimensional domain)}
\begin{multline}
\label{est:604}
M(k)
\left(
\int_{\Omega}
| (S({p}_{n}^{m}) - S(k))_{-} |^{q}
{\,{\rm d}x}
\right)^{2/q}
\\
\leq
C_{E}
M(k)
\int_{\Omega}
|\nabla (S({p}_{n}^{m}) - S(k))_{-}|^2
{\,{\rm d}x},
\end{multline}
where {$2<q<+\infty$} and, here, the embedding constant $C_{E}$ depends
only on $\Omega$.
We may now combine
\eqref{est:602}--\eqref{est:604}
to obtain
\begin{multline}
\label{est:605}
M(k)
\left(
\int_{\Omega}
| (S({p}_{n}^{m}) - S(k))_{-} |^{q}
{\,{\rm d}x}
\right)^{2/q}
\\
\leq
C_{E}
(C_{\phi}S_s + |\alpha_1|)
\int_{\Omega}
|f({x},{{p}_{n}^{m}}, {c}_{n}^{m-1}, \vartheta_{n}^{m-1}, {r}_{n}^{m-1})
(S({p}_{n}^{m}) - S(k))_{-}|
{\,{\rm d}x}.
\end{multline}
Taking into account \eqref{assum:bound_f} we have
\begin{multline}
\label{est:606}
\int_{\Omega}
|f({x},{{p}_{n}^{m}}, {c}_{n}^{m-1}, \vartheta_{n}^{m-1}, {r}_{n}^{m-1})
(S({p}_{n}^{m}) - S(k))_{-}|
{\,{\rm d}x}
\\
\leq
\tilde{f}({k})
\int_{\Omega}
|(S({p}_{n}^{m}) - S(k))_{-}|
{\,{\rm d}x}
\end{multline}
and
applying the H\"{o}lder's inequality to the right hand side in \eqref{est:606}
and combining \eqref{est:605} and \eqref{est:606}
we arrive at the estimate
\begin{multline}
\label{est:607}
\left(
\int_{A_k(m)}
| (S({p}_{n}^{m}) - S(k))_{-} |^{q}
{\,{\rm d}x}
\right)^{1/q}
\\
\leq
\frac{C_{E}
(C_{\phi}S_s + |\alpha_1|)\tilde{f}({k})}{M(k)}
\left(
\int_{A_k(m)}
1
{\,{\rm d}x}
\right)^{1/q'},
\end{multline}
where
\begin{equation*}
A_k(m) =
\left\{
{x} \in \Omega;
\;
{p}_{n}^{m}(x) < k
\right\}.
\end{equation*}
On the other hand,
applying the H\"{o}lder's inequality we have
\begin{align*}
&
\int_{A_k(m)}
| (S({p}_{n}^{m}) - S(k))_{-} |^{2}
{\,{\rm d}x}
\nonumber
\\
\leq
&
\left(
\int_{A_k(m)}
| (S({p}_{n}^{m}) - S(k))_{-} |^{q}
{\,{\rm d}x}
\right)^{2/q}
\left(
\int_{A_k(m)}
1
{\,{\rm d}x}
\right)^{(q-2)/q}
\end{align*}
and using \eqref{est:607} yields
\begin{align}
\label{est:609}
&
\int_{A_k(m)}
| (S({p}_{n}^{m}) - S(k))_{-} |^{2}
{\,{\rm d}x}
\nonumber
\\
\leq
&
\left(
\frac{C_{E}
(C_{\phi}S_s + |\alpha_1|)\tilde{f}({k})}{M(k)}
\right)^{2}
\left(
\int_{A_k(m)}
1
{\,{\rm d}x}
\right)^{(3q-4)/q}.
\end{align}
In view of \eqref{est:600a}
and employing \eqref{bound_phi}
we can write
\begin{multline}
\label{est:610}
\phi_1
\int_{A_k(i)}
| (S({p}_{n}^{i}) - S(k))_{-} |^{2}
{\,{\rm d}x}
\leq
I_{k}(i)
\\
\leq
I_{k}(m)
\leq
\phi_2
\int_{A_k(m)}
| (S({p}_{n}^{m}) - S(k))_{-} |^{2}
{\,{\rm d}x}
\end{multline}
for all $ i=1,2,\dots,n$.
Since $\ell < k$ implies $A_{\ell}(i)  \subset A_k(i)$,
and
$(S({p}_{n}^{i}) - S(k))_{-} \leq (S(\ell) - S(k))<0$ on $A_{\ell}(i)$,
we have
$$
|S(\ell) - S(k))|^2 \leq |(S({p}_{n}^{i}) - S(k))_{-}|^2
\quad  \textmd{ on }
A_{\ell}(i)
$$
and thus
\begin{align}
\label{est:611}
(S(\ell) - S(k))^2 |A_{\ell}(i)|
&
\leq
\int_{A_{\ell}(i)}
| (S({p}_{n}^{i}) - S(k))_{-} |^{2}
{\,{\rm d}x}
\nonumber
\\
&
\leq
\int_{A_k(i)}
| (S({p}_{n}^{i}) - S(k))_{-} |^{2}
{\,{\rm d}x}
\nonumber
\\
&
\leq
\frac{\phi_2}{\phi_1}
\int_{A_k(m)}
| (S({p}_{n}^{m}) - S(k))_{-} |^{2}
{\,{\rm d}x}.
\end{align}
Finally, from this and \eqref{est:609} we deduce
\begin{equation}
\label{est:612}
|A_{\ell}(i)|
\leq
\frac{\phi_2}{\phi_1}
\left(
\frac{C_{E}
(C_{\phi}S_s + |\alpha_1|)\tilde{f}({k})}{M(k)(S(\ell) - S(k))}
\right)^{2}
|A_{k}(m)|^{(3q-4)/q}
\quad
\textmd{ for all } i=1,\dots,n.
\end{equation}
To conclude the proof of Theorem~\ref{thm:max_principle_pressure}, we define
\begin{equation*}
\mu_k = \max_{i=1,\dots,n} |A_{k}(i)|.
\end{equation*}
Now, \eqref{est:612} implies  (recall $\ell < k$)
\begin{equation}
\label{est:613}
|\mu_{\ell}|
\leq
\frac{\phi_2}{\phi_1}
\left(
\frac{C_{E}
(C_{\phi}S_s + |\alpha_1|)\tilde{f}({k})}{M(k)(S(\ell) - S(k))}
\right)^{2}
|\mu_{k}|^{(3q-4)/q}.
\end{equation}
Next, we are going to apply Lemma~\ref{lem:iteration}.
In particular,
we define a decreasing sequence
\begin{equation*}
d_{j} = \frac{\delta}{2} \left( 1 + \frac{1}{2^{j}} \right),
\qquad
j=0,1,2,\dots,
\end{equation*}
where $\delta$ is a small positive real number
and let
\begin{equation*}
k_{j} = S^{-1} (d_{j}).
\end{equation*}
Since $S$ is strictly increasing function,
it is clear that
\begin{displaymath}
k_{j+1} < k_{j},
\qquad
j=0,1,2,\dots.
\end{displaymath}
Then from \eqref{est:613} we have
\begin{align}
\label{est:615}
Z_{j+1}
&
\leq
\frac{\phi_2}{\phi_1}
\left(
\frac{C_{E}
(C_{\phi}S_s + |\alpha_1|)\tilde{f}({S^{-1} (d_{j})})}{M(S^{-1} (d_{j}))(d_{j+1} - d_{j})}
\right)^{2}
Z_{j}^{(3q-4)/q}
\nonumber
\\
&
\leq
\frac{\phi_2}{\phi_1}
\left(
\frac{C_{E}
(C_{\phi}S_s + |\alpha_1|)\tilde{f}({S^{-1} (\delta)}) 4}{M({S^{-1} (\delta)})\delta}
\right)^{2}
4^{j}
Z_{j}^{[1+2(q-2)/q]},
\end{align}
where
\begin{equation}
Z_{j}  =  |\mu_{k_{j}}|.
\end{equation}
Recall that $2 < q < +\infty$. In view of Assumption (v) and taking $\delta >0$ ``small enough''
we apply Lemma~\ref{lem:iteration} to get \eqref{limit_00}.
Explicitly, this means that there exists $\ell$ (independent of $n$) such that
$$
|\mu_{\ell}| =0,
$$
in other words,
\begin{equation*}
{p}_{n}^{i} \geq \ell \; \textmd{ almost everywhere in } \Omega
\textmd{ and for all } \; i=1,\dots,n.
\end{equation*}
The proof of Theorem~\ref{thm:max_principle_pressure} is complete.
\end{proof}

\bigskip

Now we are ready to prove Theorem~\ref{thm:aprox}.

\begin{proof}[Proof of Theorem~\ref{thm:aprox}]
The proof rests on the
$W^{1,s}$-regularity of elliptic problems presented in \cite{gallouet,groger1989}
and the embedding
$W_{\Gamma_D}^{1,s}(\Omega) \subset L^{\infty}(\Omega)$ if $s>2$
(recall that  $\Omega$ is a bounded domain in $\mathbb{R}^2$).

We begin by proving the existence of ${p}_{n}^{i}  \in {W_{\Gamma_D}^{1,2}(\Omega)}$,
being the solution to problem \eqref{approximate_problem_01}.
Due to Theorem~\ref{thm:max_principle_pressure}, we may consider the truncated
function $\widetilde{k}_r$ defined by
\begin{equation*}
\widetilde{k}_r (\xi)
\equiv
\left\{
\begin{array}{ll}
{k}_r (S(\xi)) ,         &  \xi > \ell ,
\\
{k}_r (S(\ell)) ,                   & \xi \leq \ell,
\end{array} \right.
\end{equation*}
where $\ell$ is taken from \eqref{est:p_uniform_bound}.
Recall that $k_{R}$ is positive and strictly increasing on $[0,S_s]$
and $S$ is positive and strictly increasing on $\mathbb{R}$. Hence
$\widetilde{k}_r$ is the increasing function such that
\begin{equation}\label{bound_rel_perm}
0 < K_0 \leq \widetilde{k}_r(\xi) \leq K_1
\qquad
\forall \xi \in \mathbb{R}
\end{equation}
with appropriate chosen constants $K_0$ and $K_1$,
say
$K_0 = {k}_r (S(\ell))$ and $K_1={k}_r (S_s)$.
Hence, problem \eqref{approximate_problem_01} takes the form
\begin{align}
\label{proof:approximate_problem_01}
&
\quad
\int_{\Omega}
\frac{k(x,{r}_{n}^{i-1}(x))}{\mu(\vartheta_{n}^{i-1}(x))}
\widetilde{k}_r({p}_{n}^{i})
\nabla {p}_{n}^{i}
\cdot\nabla\zeta
{\,{\rm d}x}
\nonumber
\\
&
\quad
+
\frac{ 1 }{h}
\int_{\Omega}
\phi({x},{r}^{i}_{n}){S}({p}^{i}_{n})\zeta
{\,{\rm d}x}
-
\int_{\Omega}
\alpha_1 f({x},{{p}_{n}^{i}},{c}_{n}^{i-1},\vartheta_{n}^{i-1},{r}_{n}^{i-1})
\zeta
{\,{\rm d}x}
\nonumber
\\
&
=
\frac{ 1 }{h}
\int_{\Omega}
\phi({x},{r}^{i-1}_{n}){S}({p}^{i-1}_{n})\zeta
{\,{\rm d}x}
\end{align}
for any $\zeta \in {W_{\Gamma_D}^{1,2}(\Omega)}$.
Note that
the unknown $r_{n}^{i}$ in the second line of \eqref{proof:approximate_problem_01}
can be easily eliminated using
the equation \eqref{approximate_problem_04a}, which can be rewritten as
\begin{equation}\label{eq:mamory_mod_01}
r_{n}^{i}
=
r_{n}^{i-1}
+
{h}
f({x},{p}_{n}^{i},{c}_{n}^{i-1},\vartheta_{n}^{i-1},{r}_{n}^{i-1}).
\end{equation}
We now
define the so called Kirchhoff transformation, which employs the primitive function
$\beta: \mathbb{R}\rightarrow\mathbb{R}$, $\zeta=\beta(\xi)$, defined by
\begin{displaymath}%
\beta(\xi) = \int\limits_{0}^{\xi} \widetilde{k}_r(s)
{\rm d}s.
\end{displaymath}
It is worth noting that \eqref{bound_rel_perm} implies $\beta$ to
be continuous and increasing, and one-to-one with $\beta^{-1}$
Lipschitz continuous.
Hence,
with the notation $u(x) = \beta({p}_{n}^{i}(x))$,
 problem \eqref{proof:approximate_problem_01}
can be rewritten in terms of a new variable $u$ as
\begin{equation}
\label{proof_exist:approximate_problem_01}
\int_{\Omega}
A({x})
\nabla {u}
\cdot\nabla\zeta
{\,{\rm d}x}
+
\int_{\Omega}
B({x},{u})
\zeta
{\,{\rm d}x}
=
\int_{\Omega}
g(x)
\zeta
{\,{\rm d}x}
\end{equation}
for any $\zeta \in {W_{\Gamma_D}^{1,2}(\Omega)}$,
where we denote briefly
\begin{align*}
A({x})
&=
\frac{k(x,{r}_{n}^{i-1}(x))}{\mu(\vartheta_{n}^{i-1}(x))},
\nonumber
\\
B({x},{u})
&=
\frac{ {S}( \beta^{-1}(u) ) }{h}
\phi({x},r_{n}^{i-1}(x)
+
{h}
f({x},\beta^{-1}(u),{c}_{n}^{i-1}(x),\vartheta_{n}^{i-1}(x),{r}_{n}^{i-1}(x)))
\nonumber
\\
&
\quad
-
\alpha_1 f({x},\beta^{-1}(u),{c}_{n}^{i-1}(x),\vartheta_{n}^{i-1}(x),{r}_{n}^{i-1}(x)),
\nonumber
\\
g({x})
&=
\frac{ \phi({x},{r}^{i-1}_{n}(x)){S}({p}^{i-1}_{n}) }{h}.
\end{align*}
Note that $g \in L^{\infty}(\Omega)$ and
\begin{align*}
& 0<A_1<A({\cdot})<A_2<+\infty \quad ( A_1, A_2 = {\rm const})
&& \textmd{a.e. in } \Omega,
\nonumber
\\
&
|B({\cdot},{\xi})| \leq C   && \forall \xi \in \mathbb{R} \textmd{ and a.e. in } \Omega.
\nonumber
\end{align*}
The existence of $u \in W_{\Gamma_D}^{1,2}(\Omega)$,
the solution of problem \eqref{proof_exist:approximate_problem_01},
follows from \cite[Chapter~2.4]{Roubicek2005}. With $u \in W_{\Gamma_D}^{1,2}(\Omega)$ in hand,
the weak maximum principle for
the problem
\begin{equation*}
\int_{\Omega}
A({x})
\nabla {u}
\cdot\nabla\zeta
{\,{\rm d}x}
=
\int_{\Omega}
g(x)
\zeta
{\,{\rm d}x}
-
\int_{\Omega}
B({x},{u})
\zeta
{\,{\rm d}x}
\end{equation*}
for any $\zeta \in {W_{\Gamma_D}^{1,2}(\Omega)}$,
gives the regularity $u \in L^{\infty}(\Omega)$,
see e.g. \cite[Chapter 4.1.2]{Wu2006}.

Note that, in view of \eqref{bound_rel_perm},
the Kirchhoff transformation preserves $L^{\infty}$ space
for the problem. We now set ${{p}_{n}^{i}}(x) := \beta^{-1} (u(x))$ a.e. in $\Omega$
to get the representation
\begin{equation*}
\nabla {{p}_{n}^{i}}
=
\frac{1}{\widetilde{k}_r( \beta^{-1}({u}))}\nabla {u},
\quad \textmd{ i.e. }
\quad
\widetilde{k}_r(  {{p}_{n}^{i}}  )\nabla {{p}_{n}^{i}}
=
\nabla {u}
\end{equation*}
and hence
\begin{equation*}
{{p}_{n}^{i}}  \in {W_{\Gamma_D}^{1,2}(\Omega)} \cap L^{\infty}(\Omega)
\quad
{ \rm iff }
\quad
u \in {W_{\Gamma_D}^{1,2}(\Omega)} \cap L^{\infty}(\Omega).
\end{equation*}
We now conclude that ${{p}_{n}^{i}}$ solves \eqref{approximate_problem_01}.

With ${{p}_{n}^{i}}  \in {W_{\Gamma_D}^{1,2}(\Omega)} \cap L^{\infty}(\Omega)$ in hand,
we rewrite the equation \eqref{approximate_problem_01} in the form
(transferring the lower-order terms to the right hand side)

\begin{align*}
&
\int_{\Omega}
a({x},{{p}_{n}^{i}},\vartheta_{n}^{i-1},{r}_{n}^{i-1})
\nabla {p}_{n}^{i}
\cdot\nabla\zeta
{\,{\rm d}x}
\nonumber
\\
=
&
\int_{\Omega}
\alpha_1 f({x},{{p}_{n}^{i}},{c}_{n}^{i-1},\vartheta_{n}^{i-1},{r}_{n}^{i-1})
\zeta
{\,{\rm d}x}
-
\int_{\Omega}
\frac{ \phi({x},{r}^{i}_{n}){S}({p}^{i}_{n}) - \phi({x},{r}^{i-1}_{n}){S}({p}^{i-1}_{n}) }{h} \zeta
{\,{\rm d}x}
\end{align*}
for any $\zeta \in {W_{\Gamma_D}^{1,2}(\Omega)}$.
In view of Assumptions (i), (iii) and (iv),
both integrals on the right hand side
make sense for any $\zeta \in W_{\Gamma_D}^{1,r'}(\Omega)$, $r'=r/(r-1)$ with some $r>2$.
Now we are able to apply \cite[Theorem~4]{gallouet} to obtain
${p}^{i}_{n} \in  W_{\Gamma_D}^{1,s}(\Omega)$ with some $s>2$.

Now with
${c}^{i-1}_{n} \in W^{1,2}(\Omega)$,
$\vartheta^{i-1}_{n} \in  W^{1,2}(\Omega)$,
${r}^{i-1}_{n} \in  W^{1,2}(\Omega)$
and
${p}^{i}_{n} \in  W_{\Gamma_D}^{1,s}(\Omega)$
(with some $s>2$) in hand,
one obtains ${r}^{i}_{n}$
directly from \eqref{eq:mamory_mod_01}.
Since $f$ is supposed to be Lipschitz continuous,
we easily deduce ${r}^{i}_{n} \in W^{1,2}(\Omega)$ (c.f. \cite[Proposition~1.28]{Roubicek2005}).
Moreover, by \eqref{assum:bound_f} we have ${r}^{i}_{n} \in L^{\infty}(\Omega)$.

The existence of ${c}^{i}_{n} \in  W_{\Gamma_D}^{1,2}(\Omega)$
and $\vartheta^{i}_{n} \in  W_{\Gamma_D}^{1,2}(\Omega)$,
being the solutions to problems
\eqref{approximate_problem_02} and \eqref{approximate_problem_03}, respectively,
can be proven in the same way as \cite[Theorem~6.5]{BenesKrupicka2016}.
In particular, with ${p}^{i}_{n} \in  W_{\Gamma_D}^{1,s}(\Omega)$, $s>2$,
and $r_{n}^{i} \in L^{\infty}(\Omega)$, given by \eqref{approximate_problem_04a},
in hand,
\eqref{approximate_problem_02} and \eqref{approximate_problem_03}
represent semilinear equations which can be solved by the approach in \cite[Chapter~2.4]{Roubicek2005}.
Analysis similar to the above yields
${c}^{i}_{n},\vartheta^{i}_{n} \in  W_{\Gamma_D}^{1,s}(\Omega)$ with some $s>2$.
By embedding theorem we have ${c}^{i}_{n},\vartheta^{i}_{n} \in L^{\infty}(\Omega)$.
The proof of Theorem~\ref{thm:aprox} is complete.
\end{proof}

\subsection{A-priori estimates for discrete approximations}
\label{sec:estimates}
In this part of the paper, which is rather technical,
we prove some uniform estimates (with respect to $n$) for
the discrete approximations of the solution.
In the following estimates, many different constants
will appear. For simplicity of notation, as above,
$C$, $C_1$, $C_2$, $\dots$, represent
generic constants which may change their numerical values
from one formula to another
but do not depend on $n$ and
the functions under consideration.

\subsubsection{Uniform bounds in  $L^{\infty}$}

We first prove the apriori $L^{\infty}$-estimate for ${c}^{i}_{n}$, $i=1,\dots,n$,
being the solution to \eqref{approximate_problem_02}.
By Theorem~\ref{thm:aprox} we have
${c}^{i}_{n} \in W_{\Gamma_D}^{1,2}(\Omega) \cap L^{\infty}(\Omega)$.
Hence, $({c}^{i}_{n})^{\ell} \in W_{\Gamma_D}^{1,2}(\Omega) \cap L^{\infty}(\Omega)$
for all $\ell=1,2,\dots$.
The following procedure is similar to that used e.g. in \cite{Filo1987,Pluschke1988} for scalar problems.
Let $\ell$ be an odd integer.
Using $\zeta = [\ell/(\ell+1)] ({c}^{i}_{n})^{\ell+1}$
as a test function in \eqref{approximate_problem_01}
and
$\eta = ({c}^{i}_{n})^{\ell}$ in \eqref{approximate_problem_02}
and combining both equations we obtain
\begin{align}\label{eq:501}
&
\quad
\frac{1}{h}
\frac{1}{\ell+1}
\int_{\Omega}
 \phi({x},{r}^{i}_{n}){S}({p}^{i}_{n})    [{c}^{i}_{n} ]^{\ell+1}
{\,{\rm d}x}
\nonumber
\\
&
\quad
-
\frac{1}{h}
\frac{1}{\ell+1}
\int_{\Omega}
   \phi({x},{r}^{i-1}_{n}){S}({p}^{i-1}_{n})   [{c}^{i-1}_{n} ]^{\ell+1}
{\,{\rm d}x}
\nonumber
\\
&
\quad
+
\frac{1}{h}
\frac{1}{\ell+1}
\int_{\Omega}
 \phi({x},{r}^{i-1}_{n}){S}({p}^{i-1}_{n})   [{c}^{i-1}_{n}]^{\ell+1}
{\,{\rm d}x}
\nonumber
\\
&
\quad
+
\frac{1}{h}
\frac{\ell}{\ell+1}
\int_{\Omega}
 \phi({x},{r}^{i-1}_{n}){S}({p}^{i-1}_{n})   [{c}^{i}_{n}]^{\ell+1}
{\,{\rm d}x}
\nonumber
\\
&
\quad
-
\frac{1}{h}
\int_{\Omega}
  \phi({x},{r}^{i-1}_{n}){S}({p}^{i-1}_{n})   {c}^{i-1}_{n}   [{c}^{i}_{n}]^{\ell}
{\,{\rm d}x}
\nonumber
\\
&
\quad
+
\int_{\Omega}
\ell [{c}^{i}_{n}]^{\ell-1}
 \phi({x},{r}^{i-1}_{n}){S}({p}^{i-1}_{n})  D_w({p}^{i-1}_{n})
\nabla {c}^{i}_{n} \cdot \nabla {c}^{i}_{n}
{\,{\rm d}x}
\nonumber
\\
&
=
0.
\end{align}
Applying the Young's inequality,
for the term in the fifth line in \eqref{eq:501} we can write
\begin{align}\label{eq:502}
&
\quad
\frac{1}{h}
\int_{\Omega}
\phi({x},{r}^{i-1}_{n}){S}({p}^{i-1}_{n})   {c}^{i-1}_{n}[{c}^{i}_{n}]^{\ell}
{\,{\rm d}x}
\nonumber
\\
&
\leq
\frac{1}{h}
\frac{1}{\ell+1}
\int_{\Omega}
\phi({x},{r}^{i-1}_{n}){S}({p}^{i-1}_{n})        [{c}^{i-1}_{n}]^{\ell+1}
{\,{\rm d}x}
\nonumber
\\
&
\quad
+
\frac{1}{h}
\frac{\ell}{\ell+1}
\int_{\Omega}
\phi({x},{r}^{i-1}_{n}){S}({p}^{i-1}_{n})        [{c}^{i}_{n}]^{\ell+1}
{\,{\rm d}x}.
\end{align}
Taking \eqref{eq:501} and \eqref{eq:502} together we deduce
\begin{align}\label{eq:503}
&
\quad
\frac{1}{h}
\frac{1}{\ell+1}
\int_{\Omega}
 \phi({x},{r}^{i}_{n}){S}({p}^{i}_{n})    [{c}^{i}_{n} ]^{\ell+1}
{\,{\rm d}x}
\nonumber
\\
&
\quad
-
\frac{1}{h}
\frac{1}{\ell+1}
\int_{\Omega}
   \phi({x},{r}^{i-1}_{n}){S}({p}^{i-1}_{n})   [{c}^{i-1}_{n} ]^{\ell+1}
{\,{\rm d}x}
\nonumber
\\
&
\quad
+
\int_{\Omega}
\ell [{c}^{i}_{n}]^{\ell-1}
 \phi({x},{r}^{i-1}_{n}){S}({p}^{i-1}_{n})  D_w({p}^{i-1}_{n})
|\nabla {c}^{i}_{n}|^2
{\,{\rm d}x}
\nonumber
\\
&
\leq
0.
\end{align}
Now, we sum \eqref{eq:503} for $i=1,\dots,j$ to get
\begin{align}\label{app:eq_bound_w_00}
&
\quad
\int_{\Omega}
 \phi({x},{r}^{j}_{n}){S}({p}^{j}_{n})[{c}^{j}_{n} ]^{\ell+1}
{\,{\rm d}x}
\nonumber
\\
&
\quad
+
h
\sum_{i=1}^{j}
\int_{\Omega}
\ell (\ell+1) [{c}^{i}_{n}]^{\ell-1}
 \phi({x},{r}^{i-1}_{n}){S}({p}^{i-1}_{n})  D_w({p}^{i-1}_{n})
|\nabla {c}^{i}_{n}|^2
{\,{\rm d}x}
\nonumber
\\
&
\leq
\int_{\Omega}
   \phi({x},{r}^{0}_{n}){S}({p}^{0}_{n})   [{c}^{0}_{n} ]^{\ell+1}
{\,{\rm d}x}.
\end{align}
Note that the second integral in
\eqref{app:eq_bound_w_00} is nonnegative ($\ell$ is supposed to be the odd integer).
Moreover, in view of \eqref{app:eq_bound_w_00}, \eqref{bound_phi}
and
\eqref{est:p_uniform_bound} we have
\begin{equation}\label{est:unform_bound_w_01}
\|{c}^{j}_{n}\|_{L^{\ell+1}(\Omega)} \leq C \|{c}_{0}\|_{L^{\ell+1}(\Omega)},
\end{equation}
where the constant $C$ is independent of $\ell$ and $n$.
Now, letting $\ell \rightarrow +\infty$ in \eqref{est:unform_bound_w_01}, we obtain
\begin{equation}\label{est:unform_bound_w_05}
\|{c}^{j}_{n}\|_{L^{\infty}(\Omega)} \leq C, \qquad j=1,\dots,n.
\end{equation}
The same $L^{\infty}$-estimate can be drawn for
temperature approximations, i.e.
\begin{equation}\label{est:unform_bound_theta_05}
\|\vartheta_{n}^{j}\|_{L^{\infty}(\Omega)} \leq C, \qquad j=1,\dots,n.
\end{equation}
Because many steps of the proof of \eqref{est:unform_bound_theta_05} are similar to those
from the preceding estimate \eqref{est:unform_bound_w_05},
we shall proceed more rapidly here, without explaining particular steps
once more.

At the same time, from \eqref{approximate_problem_04a} and using \eqref{assum:bound_f}
we have
\begin{multline}\label{est:unform_bound_r_01}
\| r_{n}^{i} \|_{L^{\infty}(\Omega)}
=
\| {h} \sum_{j=1}^{i}
f({x},{p}_{n}^{j},{c}_{n}^{j-1},\vartheta_{n}^{j-1},{r}_{n}^{j-1}) \|_{L^{\infty}(\Omega)}
\\
\leq
ih C_{f}
\leq
T C_{f},
\quad
i=1,\dots,n.
\end{multline}

{\subsubsection{Energy estimates for discrete approximations of primary unknowns}}

We start with the uniform estimate for pressure approximations ${p}_{n}^{i}$.
We test \eqref{approximate_problem_01} with $\zeta = {p}_{n}^{i}$ to get
\begin{align}
\label{est:press_01}
&
\quad
\int_{\Omega}
\left[\phi(x,r_{n}^{i})-\phi(x,r_{n}^{i-1}) \right]{S}({p}^{i}_{n}){p}^{i}_{n}
{\,{\rm d}x}
\nonumber
\\
&
\quad
+
\int_{\Omega}
\phi(x,r_{n}^{i-1})
\left[ {S}({p}^{i}_{n}) - {S}({p}^{i-1}_{n})  \right]
{p}_{n}^{i}
{\,{\rm d}x}
\nonumber
\\
&
\quad
+
{h}
\int_{\Omega}
a({x},{{p}_{n}^{i}},\vartheta_{n}^{i-1},{r}_{n}^{i-1})
\nabla {p}_{n}^{i}
\cdot\nabla {p}_{n}^{i}
{\,{\rm d}x}
\nonumber
\\
&
=
{h}
\int_{\Omega}
\alpha_1 f({x},{{p}_{n}^{i}},{c}_{n}^{i-1},\vartheta_{n}^{i-1},{r}_{n}^{i-1})
{p}_{n}^{i}
{\,{\rm d}x}.
\end{align}
Define the function $\Theta_{S}: \mathbb{R} \rightarrow \mathbb{R}$ given by the equation
\begin{equation}\label{est:press_02}
\Theta_{S}(\xi) = \int_{0}^{\xi}   S'(z)z  {{\rm d}z},
\qquad
\xi \in \mathbb{R}.
\end{equation}
It is easy to check that
\begin{equation}\label{est:press_03}
\Theta_{S}(\xi_1) - \Theta_{S}(\xi_2) \leq [S(\xi_1) - S(\xi_2)]\xi_1
\qquad
\forall \xi_1,\xi_2 \in \mathbb{R}.
\end{equation}
Using the inequality \eqref{est:press_03} in the equation
\eqref{est:press_01} we arrive at
\begin{align*}
&
\quad
\int_{\Omega}
\left(\phi(x,r_{n}^{i})-\phi(x,r_{n}^{i-1}) \right)
\left[
{S}({p}^{i}_{n}){p}^{i}_{n} - \Theta_{S}({p}^{i}_{n})
\right]
{{\rm d}x}
\nonumber
\\
&
\quad
+
\int_{\Omega}
\left[
\phi(x,r_{n}^{i})\Theta_{S}({p}^{i}_{n}) - \phi(x,r_{n}^{i-1})\Theta_{S}({p}^{i-1}_{n})
\right]
{{\rm d}x}
\nonumber
\\
&
\quad
+
{h}
\int_{\Omega}
a({x},{{p}_{n}^{i}},\vartheta_{n}^{i-1},{r}_{n}^{i-1})
\nabla {p}_{n}^{i}
\cdot
\nabla {p}_{n}^{i}
{{\rm d}x}
\nonumber
\\
&
\leq
{h}
\int_{\Omega}
\alpha_1 f({x},{{p}_{n}^{i}},{c}_{n}^{i-1},\vartheta_{n}^{i-1},{r}_{n}^{i-1})
{p}_{n}^{i}
{{\rm d}x}.
\end{align*}
From this we have
\begin{align}
\label{est:press_05}
&
\quad
\int_{\Omega}
\left[
\phi(x,r_{n}^{i})\Theta_{S}({p}^{i}_{n}) - \phi(x,r_{n}^{i-1})\Theta_{S}({p}^{i-1}_{n})
\right]
{{\rm d}x}
\nonumber
\\
&
\quad
+
{h}
\int_{\Omega}
a({x},{{p}_{n}^{i}},\vartheta_{n}^{i-1},{r}_{n}^{i-1})
|\nabla {p}_{n}^{i}|^2
{{\rm d}x}
\nonumber
\\
&
\leq
{h}
\int_{\Omega}
\alpha_1 f({x},{{p}_{n}^{i}},{c}_{n}^{i-1},\vartheta_{n}^{i-1},{r}_{n}^{i-1})
{p}_{n}^{i}
{{\rm d}x}
\nonumber
\\
&
\quad
-
\int_{\Omega}
\left(\phi(x,r_{n}^{i})-\phi(x,r_{n}^{i-1}) \right)
\left[
{S}({p}^{i}_{n}){p}^{i}_{n} - \Theta_{S}({p}^{i}_{n})
\right]
{{\rm d}x}.
\end{align}
Taking into account \eqref{lipschitz_phi} and \eqref{assum:bound_f}
and
estimating the right-hand side in \eqref{est:press_05} we deduce
\begin{align}
\label{est:press_06}
&
{h}
\int_{\Omega}
\alpha_1 f({x},{{p}_{n}^{i}},{c}_{n}^{i-1},\vartheta_{n}^{i-1},{r}_{n}^{i-1})
{p}_{n}^{i}
{{\rm d}x}
\nonumber
\\
&
-
\int_{\Omega}
\left(\phi(x,r_{n}^{i})-\phi(x,r_{n}^{i-1}) \right)
\left[
{S}({p}^{i}_{n}){p}^{i}_{n} - \Theta_{S}({p}^{i}_{n})
\right]
{{\rm d}x}
\nonumber
\\
\leq
&
\;
{h}
\int_{\Omega}
|\alpha_1
f({x},{{p}_{n}^{i}},{c}_{n}^{i-1},\vartheta_{n}^{i-1},{r}_{n}^{i-1})
{p}_{n}^{i}|
{{\rm d}x}
\nonumber
\\
&
+
\int_{\Omega}
C_{\phi} | r_{n}^{i} - r_{n}^{i-1} |
|{S}({p}^{i}_{n}){p}^{i}_{n} - \Theta_{S}({p}^{i}_{n})|
{{\rm d}x}
\nonumber
\\
\leq
&
\;
|\alpha_1|  {h}
\int_{\Omega}
|f({x},{{p}_{n}^{i}},{c}_{n}^{i-1},\vartheta_{n}^{i-1},{r}_{n}^{i-1})|
|{p}_{n}^{i}|
{{\rm d}x}
\nonumber
\\
&
+
{h}
\int_{\Omega}
C_{\phi}
|f({x},{{p}_{n}^{i}},{c}_{n}^{i-1},\vartheta_{n}^{i-1},{r}_{n}^{i-1})|
|{S}({p}^{i}_{n}){p}^{i}_{n} - \Theta_{S}({p}^{i}_{n})|
{{\rm d}x}
\nonumber
\\
\leq
&
\;
C_1 {h}
\int_{\Omega}
|{p}^{i}_{n}|
{{\rm d}x}
+
C_2 {h}
\int_{\Omega}
\Theta_{S}({p}^{i}_{n})
{{\rm d}x}.
\end{align}
Combining \eqref{est:press_05} and \eqref{est:press_06},
using \eqref{est:p_uniform_bound}
and applying the Young's inequality
to the first term on the right-hand side in
\eqref{est:press_06}
we get
\begin{align}
\label{est:press_07}
&
\int_{\Omega}
\left[
\phi(x,r_{n}^{i})\Theta_{S}({p}^{i}_{n}) - \phi(x,r_{n}^{i-1})\Theta_{S}({p}^{i-1}_{n})
\right]
{{\rm d}x}
+
C_1
{h}
\int_{\Omega}
|\nabla {p}_{n}^{i}|^2
{{\rm d}x}
\nonumber
\\
\leq
&
\;
C_2 {h}
+
C_3 {h}
\int_{\Omega}
\Theta_{S}({p}^{i}_{n})
{{\rm d}x}.
\end{align}
Sum \eqref{est:press_07} for $i=1,2,\dots,k$.
We have
\begin{align}
\label{est:press_08}
&
\int_{\Omega}
\phi(x,r_{n}^{k})\Theta_{S}({p}^{k}_{n})
{{\rm d}x}
+
C_1
{h}
\sum_{i=1}^{k}
\int_{\Omega}
|\nabla {p}_{n}^{i}|^2
{{\rm d}x}
\nonumber
\\
\leq
&
\;
\int_{\Omega}
\phi(x,r_{n}^{0})\Theta_{S}({p}^{0}_{n})
{{\rm d}x}
+
C_2 k {h}
+
C_3 {h}
\sum_{i=1}^{k}
\int_{\Omega}
\Theta_{S}({p}^{i}_{n})
{{\rm d}x},
\qquad k=1,2,\dots,n.
\end{align}
We now apply the discrete version of the Gronwall's inequality
(see e.g. \cite[Chapter~1, inequality (1.67)]{Roubicek2005}) to get
\begin{equation}\label{est:energy_pressure}
\int_{\Omega}
\Theta_{S}({p}^{k}_{n})
{{\rm d}x}
+
{h}
\sum_{i=1}^{k}
\int_{\Omega}
|\nabla {p}_{n}^{i}|^2
{{\rm d}x}
\leq
C,
\qquad
k=1,2,\dots,n.
\end{equation}
%
%
%
In what follows,
we proceed by proving a similar uniform estimate for approximations of $c$.
Using $\eta = 2 {c}^{i}_{n} $
as a test function in \eqref{approximate_problem_02} we have
\begin{align}\label{est:_w_01}
&
\quad
\int_{\Omega}
\phi({x},{r}^{i}_{n}){S}({p}^{i}_{n})({c}^{i}_{n})^2
-
\phi({x},{r}^{i-1}_{n}){S}({p}^{i-1}_{n})({c}^{i-1}_{n})^2
{\,{\rm d}x}
\nonumber
\\
&
\quad
+
\int_{\Omega}
\left[
\phi({x},{r}^{i}_{n}){S}({p}^{i}_{n})
-
\phi({x},{r}^{i-1}_{n}){S}({p}^{i-1}_{n})
\right]({c}^{i}_{n})^2
{\,{\rm d}x}
\nonumber
\\
&
\quad
+
\int_{\Omega}
\phi({x},{r}^{i-1}_{n}){S}({p}^{i-1}_{n})
\left(
{c}^{i}_{n}-{c}^{i-1}_{n}
\right)^2
{\,{\rm d}x}
\nonumber
\\
&
\quad
+
2 h
\int_{\Omega}
\phi({x},{r}^{i}_{n}){S}({p}^{i}_{n})  D_w(x,{p}_{n}^{i})
\nabla {c}^{i}_{n} \cdot \nabla {c}^{i}_{n}
{\,{\rm d}x}
\nonumber
\\
&
\quad
+
h
\int_{\Omega}
a({x},{{p}_{n}^{i}},\vartheta_{n}^{i-1},{r}_{n}^{i-1})
\nabla {p}_{n}^{i}
\cdot
2{c}^{i}_{n} \nabla {c}^{i}_{n}
{\,{\rm d}x}
\nonumber
\\
&
=0.
\end{align}
One is allowed to use $\zeta =  ({c}^{i}_{n})^2$
as a test function in \eqref{approximate_problem_01} to obtain
\begin{align}\label{est:_w_02}
&
\quad
\int_{\Omega}
\left[
\phi({x},{r}^{i}_{n}){S}({p}^{i}_{n})
-
\phi({x},{r}^{i-1}_{n}){S}({p}^{i-1}_{n})
\right]({c}^{i}_{n})^2
{\,{\rm d}x}
\nonumber
\\
&
\quad
+
h
\int_{\Omega}
a({x},{{p}_{n}^{i}},\vartheta_{n}^{i-1},{r}_{n}^{i-1})
\nabla {p}_{n}^{i}
\cdot\nabla ({c}^{i}_{n})^2
{\,{\rm d}x}
\nonumber
\\
&
=
h
\int_{\Omega}
\alpha_1 f({x},{{p}_{n}^{i}},{c}_{n}^{i-1},\vartheta_{n}^{i-1},{r}_{n}^{i-1})
 ({c}^{i}_{n})^2
{\,{\rm d}x}.
\end{align}
Subtracting \eqref{est:_w_02} from \eqref{est:_w_01} gives
\begin{align}\label{est:w_100}
&
\quad
\int_{\Omega}
\phi({x},{r}^{i}_{n}){S}({p}^{i}_{n})({c}^{i}_{n})^2
-
\phi({x},{r}^{i-1}_{n}){S}({p}^{i-1}_{n})({c}^{i-1}_{n})^2
{\,{\rm d}x}
\nonumber
\\
&
\quad
+
\int_{\Omega}
\phi({x},{r}^{i-1}_{n}){S}({p}^{i-1}_{n})
\left(
{c}^{i}_{n}-{c}^{i-1}_{n}
\right)^2
{\,{\rm d}x}
\nonumber
\\
&
\quad
+
2 h
\int_{\Omega}
\phi({x},{r}^{i}_{n}){S}({p}^{i}_{n})  D_w(x,{p}_{n}^{i})
\nabla {c}^{i}_{n} \cdot \nabla {c}^{i}_{n}
{\,{\rm d}x}
\nonumber
\\
&
\quad
+
h
\int_{\Omega}
\alpha_1 f({x},{{p}_{n}^{i}},{c}_{n}^{i-1},\vartheta_{n}^{i-1},{r}_{n}^{i-1})
 ({c}^{i}_{n})^2
{\,{\rm d}x}
\nonumber
\\
&
=0.
\end{align}
Upon addition \eqref{est:w_100} for $i=1,2,\dots,j$ and taking into account
\eqref{assum:bound_f},
we can write
\begin{multline}\label{est:w_151}
\int_{\Omega}
\phi({x},{r}^{j}_{n}){S}({p}^{j}_{n})({c}^{j}_{n})^2
{\,{\rm d}x}
\\
\leq
\phi({x},{r}^{0}_{n}){S}({p}^{0}_{n})({c}^{0}_{n})^2
+
h |\alpha_1| C_f
\sum_{i=1}^{j}
\int_{\Omega}
 ({c}^{i}_{n})^2
{\,{\rm d}x},
\quad
j=1,\dots,n.
\end{multline}
Noting that, in view of \eqref{est:p_uniform_bound}, \eqref{con11a},
\eqref{con12b} and \eqref{bound_phi}, there exists a positive constant $C$
(independent of $n$) such that
\begin{align}\label{bound_below_01}
\phi({x},{r}^{i}_{n}), \;
{S}({p}^{i}_{n}), \;
D_w(x,{p}_{n}^{i}) > C \qquad \textmd{ in } \Omega, \quad i=1,\dots,n,
\end{align}
the inequality \eqref{est:w_151} can be simplified to
\begin{equation}\label{est:w_152}
\int_{\Omega}
({c}^{j}_{n})^2
{\,{\rm d}x}
\leq
C_1
+
C_2 h
\sum_{i=1}^{j}
\int_{\Omega}
 ({c}^{i}_{n})^2
{\,{\rm d}x},
\quad
j=1,\dots,n.
\end{equation}
Now, similarly as in \eqref{est:press_08},
we can use the Gronwall's inequality. By doing that, in view of \eqref{est:w_100},
\eqref{bound_below_01} and \eqref{est:w_152},
we obtain the estimate
\begin{equation}\label{est:energy_concentration}
\max_{i=1,\dots,n}
\int_{\Omega}
|{c}^{i}_{n}|^2
{\,{\rm d}x}
+
h
\sum_{i=1}^{n}
\int_{\Omega}
|\nabla {c}^{i}_{n}|^2
{\,{\rm d}x}
 \leq
C.
\end{equation}
%
%
%
The same uniform estimate can be drawn for the temperature approximations $\vartheta^{i}_{n}$.
We use $\psi  = 2\vartheta^{i}_{n}$
as a test function in \eqref{approximate_problem_03} to obtain
\begin{align*}
&
\quad
\int_{\Omega}
\frac{
\phi({x},{r}^{i}_{n}){S}({p}^{i}_{n})   \vartheta^{i}_{n}
-
\phi({x},{r}^{i-1}_{n}){S}({p}^{i-1}_{n})   \vartheta_{n}^{i-1} }{h}
2\vartheta^{i}_{n}
{\,{\rm d}x}
\nonumber
\\
&
\quad
+
\int_{\Omega}
\frac{ \varrho(x,r_{n}^{i})\vartheta_{n}^{i} - \varrho(x,r_{n}^{i-1})\vartheta_{n}^{i-1} }{h}  2\vartheta^{i}_{n}
{\,{\rm d}x}
\nonumber
\\
&
\quad
+
2\int_{\Omega}
\lambda({x},{{p}_{n}^{i-1}},\vartheta_{n}^{i-1},{r}_{n}^{i-1})
| \nabla \vartheta_{n}^{i} |^2
{\,{\rm d}x}
\nonumber
\\
&
\quad
+
\int_{\Omega}
a({x},{{p}_{n}^{i}},\vartheta_{n}^{i-1},{r}_{n}^{i-1})
\nabla {p}_{n}^{i}
\cdot 2\vartheta_{n}^{i} \nabla\vartheta_{n}^{i}
{\,{\rm d}x}
\nonumber
\\
&
=
2
\int_{\Omega}
\alpha_2
f({x},{{p}_{n}^{i}},{c}_{n}^{i-1},\vartheta_{n}^{i-1},{r}_{n}^{i-1})
\vartheta_{n}^{i}
{\,{\rm d}x}.
\end{align*}
The above equation may be written as
\begin{align}
\label{est:301}
&
\quad
\int_{\Omega}
\frac{
\phi({x},{r}^{i}_{n}){S}({p}^{i}_{n})   [\vartheta^{i}_{n}]^{2}
-
\phi({x},{r}^{i-1}_{n}){S}({p}^{i-1}_{n})   [\vartheta_{n}^{i-1}]^{2}
}{h}
{\,{\rm d}x}
\nonumber
\\
&
\quad
+
\int_{\Omega}
\frac{
\phi({x},{r}^{i}_{n}){S}({p}^{i}_{n})
-
\phi({x},{r}^{i-1}_{n}){S}({p}^{i-1}_{n})
}{h}
[\vartheta^{i}_{n}]^{2}
{\,{\rm d}x}
\nonumber
\\
&
\quad
+
\int_{\Omega}
\frac{ \varrho(x,r_{n}^{i})[\vartheta_{n}^{i}]^2 - \varrho(x,r_{n}^{i-1})[\vartheta_{n}^{i-1}]^2 }{h}
{\,{\rm d}x}
+
\int_{\Omega}
\frac{ \varrho(x,r_{n}^{i})- \varrho(x,r_{n}^{i-1})}{h}
[\vartheta_{n}^{i}]^2 {\,{\rm d}x}
\nonumber
\\
&
\quad
+
\int_{\Omega}
\frac{ \phi({x},{r}^{i-1}_{n}){S}({p}^{i-1}_{n}) [\vartheta^{i}_{n} - \vartheta^{i-1}_{n}]^{2}  }{h}
{\,{\rm d}x}
+
\int_{\Omega}
\frac{ \varrho(x,r_{n}^{i-1})[\vartheta_{n}^{i} - \vartheta^{i-1}_{n}]^{2}  }{h}
{\,{\rm d}x}
\nonumber
\\
&
\quad
+
2\int_{\Omega}
\lambda({x},{{p}_{n}^{i-1}},\vartheta_{n}^{i-1},{r}_{n}^{i-1})
| \nabla \vartheta_{n}^{i} |^2
{\,{\rm d}x}
\nonumber
\\
&
\quad
+
\int_{\Omega}
a({x},{{p}_{n}^{i}},\vartheta_{n}^{i-1},{r}_{n}^{i-1})
\nabla {p}_{n}^{i}
\cdot \nabla [\vartheta_{n}^{i}]^2
{\,{\rm d}x}
\nonumber
\\
&
=
2
\int_{\Omega}
\alpha_2
f({x},{{p}_{n}^{i}},{c}_{n}^{i-1},\vartheta_{n}^{i-1},{r}_{n}^{i-1})
\vartheta_{n}^{i}
{\,{\rm d}x}.
\end{align}
Putting $\zeta = [\vartheta_{n}^{i}]^2$ into \eqref{approximate_problem_01},
we get
\begin{align}
\label{est:302}
&
\quad
\int_{\Omega}
\frac{ \phi({x},{r}^{i}_{n}){S}({p}^{i}_{n})
-
\phi({x},{r}^{i-1}_{n}){S}({p}^{i-1}_{n})
}{h} [\vartheta_{n}^{i}]^2
{\,{\rm d}x}
\nonumber
\\
&
\quad
+
\int_{\Omega}
a({x},{{p}_{n}^{i}},\vartheta_{n}^{i-1},{r}_{n}^{i-1})
\nabla {p}_{n}^{i}
\cdot\nabla [\vartheta_{n}^{i}]^2
{\,{\rm d}x}
\nonumber
\\
&
=
\int_{\Omega}
\alpha_1 f({x},{{p}_{n}^{i}},{c}_{n}^{i-1},\vartheta_{n}^{i-1},{r}_{n}^{i-1})
[\vartheta_{n}^{i}]^2
{\,{\rm d}x}.
\end{align}
Substituting \eqref{est:302} into \eqref{est:301},
multiplying by ${h}$
and taking into account \eqref{approximate_problem_04a} we deduce
\begin{align}
\label{eq:201}
&
\quad
\int_{\Omega}
\left(
\phi({x},{r}^{i}_{n}){S}({p}^{i}_{n})   [\vartheta^{i}_{n}]^{2}
-
\phi({x},{r}^{i-1}_{n}){S}({p}^{i-1}_{n})   [\vartheta_{n}^{i-1}]^{2}
\right)
{\,{\rm d}x}
\nonumber
\\
&
\quad
+
\int_{\Omega}
\left(
\varrho(x,r_{n}^{i})[\vartheta_{n}^{i}]^2 - \varrho(x,r_{n}^{i-1})[\vartheta_{n}^{i-1}]^2
\right)
{\,{\rm d}x}
\nonumber
\\
&
\quad
+
2\int_{\Omega}
\lambda({x},{{p}_{n}^{i-1}},\vartheta_{n}^{i-1},{r}_{n}^{i-1})
| \nabla \vartheta_{n}^{i} |^2
{\,{\rm d}x}
\nonumber
\\
&
\leq
{h}
\int_{\Omega}
\left(
2 |\alpha_2| |\vartheta_{n}^{i}|
+
|\alpha_1| [\vartheta_{n}^{i}]^2
+
[\vartheta_{n}^{i}]^2
\right)
 |f({x},{{p}_{n}^{i}},{c}_{n}^{i-1},\vartheta_{n}^{i-1},{r}_{n}^{i-1})|
{\,{\rm d}x}.
\end{align}
We now apply the estimate
\begin{multline*}
{h}
\int_{\Omega}
\left(
2 |\alpha_2| |\vartheta_{n}^{i}|
+
|\alpha_1| [\vartheta_{n}^{i}]^2
+
[\vartheta_{n}^{i}]^2
\right)
 |f({x},{{p}_{n}^{i}},{c}_{n}^{i-1},\vartheta_{n}^{i-1},{r}_{n}^{i-1})|
{\,{\rm d}x}
\\
\leq
{h}
\int_{\Omega}
\left(
C_1
+
C_2 [\vartheta_{n}^{i}]^2
\right)
C_{f}
{\,{\rm d}x},
\end{multline*}
which holds for ``sufficiently large'' $C_1$,
and sum \eqref{eq:201} for $i = 1,. . . ,j$ to obtain
\begin{align*}
&
\quad
\int_{\Omega}
\left(
\phi({x},{r}^{j}_{n}){S}({p}^{j}_{n}) + \varrho(x,r_{n}^{j})
\right)
[\vartheta^{j}_{n}]^{2}
{\,{\rm d}x}
\nonumber
\\
&
\quad
+
2 {h}
\sum_{i=1}^{j}
\int_{\Omega}
\lambda({x},{{p}_{n}^{i-1}},\vartheta_{n}^{i-1},{r}_{n}^{i-1})
| \nabla \vartheta_{n}^{i} |^2
{\,{\rm d}x}
\nonumber
\\
&
\leq
\int_{\Omega}\left(
\phi({x},{r}^{0}_{n}){S}({p}^{0}_{n}) + \varrho(x,r_{n}^{0})
\right)
[\vartheta^0_{n}]^{2}
{\,{\rm d}x}
\nonumber
\\
&
\quad
+
{h}
\sum_{i=1}^{j}
\int_{\Omega}
\left(
C_1
+
C_2 [\vartheta_{n}^{i}]^2
\right)
 C_{f}
{\,{\rm d}x}.
\end{align*}
By \eqref{cond:density}, \eqref{con12c}, \eqref{bound_phi}
and \eqref{est:p_uniform_bound},
the above estimate may be simplified as
\begin{equation*}
\int_{\Omega}
|\vartheta^{j}_{n}|^{2}
{\,{\rm d}x}
+
{h}
\sum_{i=1}^{j}
\int_{\Omega}
| \nabla \vartheta_{n}^{i} |^2
{\,{\rm d}x}
\leq
C_1
+
C_2
T
+
C_3
{h}
\sum_{i=1}^{j}
\int_{\Omega}
|\vartheta_{n}^{i}|^2
{\,{\rm d}x}.
\end{equation*}
As before,
we can now use the Gronwall's inequality.
As a consequence, we obtain
\begin{equation}
\label{est:energy_temperature}
\int_{\Omega}
|\vartheta^{j}_{n}|^{2}
{\,{\rm d}x}
+
{h}
\sum_{i=1}^{j}
\int_{\Omega}
| \nabla \vartheta_{n}^{i} |^2
{\,{\rm d}x}
\leq
C,
\qquad
j=1,2,\dots,n.
\end{equation}
Finally,
using Assumptions (iv) and (vi), \eqref{est:energy_pressure},
\eqref{est:energy_concentration}
and
\eqref{est:energy_temperature}
and applying \cite[Proposition~1.28]{Roubicek2005},
the uniform estimate
\begin{equation}
\label{est:energy_memory_function}
{h}
\sum_{i=1}^{j}
\int_{\Omega}
| \nabla {r}_{n}^{i} |^2
{\,{\rm d}x}
\leq
C,
\qquad
j=1,2,\dots,n,
\end{equation}
can be obtained directly from \eqref{approximate_problem_04a}.
Note that
\eqref{est:energy_memory_function} together with
\eqref{est:unform_bound_r_01}
yields
\begin{equation}
\label{est:energy_memory_function_b}
{h}
\sum_{i=1}^{j}
\int_{\Omega}
\|{r}_{n}^{i} \|^2_{W^{1,2}(\Omega)}
{\,{\rm d}x}
\leq
C,
\qquad
j=1,2,\dots,n.
\end{equation}
Moreover, from \eqref{approximate_problem_04a}
(see also \eqref{eq:mamory_mod_01})
one observes immediately that
\begin{equation*}
\frac{r_{n}^{i} - r_{n}^{i-1}}{h}
=
f({x},{p}_{n}^{i},{c}_{n}^{i-1},\vartheta_{n}^{i-1},{r}_{n}^{i-1})
\end{equation*}
and therefore,
in view of \eqref{assum:bound_f}, we have
\begin{equation}\label{eq:mamory_mod_04}
\left\|
\frac{r_{n}^{i} - r_{n}^{i-1}}{h}
\right\|_{L^{\infty}(\Omega)}
\leq
C_{f},
\quad
i=1,2,\dots,n.
\end{equation}


\bigskip

\subsubsection{Further estimates}
We now derive additional auxiliary estimates which will be used in the following section.
Such estimates play the crucial role in compactness arguments
(see \cite[Lemma~1.9]{AltLuckhaus1983})
and taking the limit $n \rightarrow +\infty$.

Let us sum up \eqref{approximate_problem_01}
for $i=j+1,\dots,j+k$ and then put
$\zeta = {p}^{j+k}_{n} - {p}^{j}_{n}$. This leads to
\begin{align}
\label{est:701}
&
\quad
\int_{\Omega}
\left[
\phi({x},{r}^{j+k}_{n}){S}({p}^{j+k}_{n}) - \phi({x},{r}^{j}_{n}){S}({p}^{j}_{n})
\right]
\left(
{p}^{j+k}_{n}
-
{p}^{j}_{n}
\right)
{\,{\rm d}x}
\nonumber
\\
&
\quad
+
h
\sum_{i=j+1}^{j+k}
\int_{\Omega}
a({x},{{p}_{n}^{i}},\vartheta_{n}^{i-1},{r}_{n}^{i-1}) \nabla {p}_{n}^{i}
\cdot\nabla\left(
{p}^{j+k}_{n}
-
{p}^{j}_{n}
\right)
{\,{\rm d}x}
\nonumber
\\
&
=
h
\sum_{i=j+1}^{j+k}
\int_{\Omega}
\alpha_1 f({x},{{p}_{n}^{i}},{c}_{n}^{i-1},\vartheta_{n}^{i-1},{r}_{n}^{i-1})
\left(
{p}^{j+k}_{n}
-
{p}^{j}_{n}
\right)
{\,{\rm d}x}.
\end{align}
From this and in view of \eqref{con12a} and \eqref{assum:bound_f} we have
\begin{align}
\label{est:702}
&
\quad
\int_{\Omega}
\left[
\phi({x},{r}^{j+k}_{n}){S}({p}^{j+k}_{n}) - \phi({x},{r}^{j}_{n}){S}({p}^{j}_{n})
\right]
\left(
{p}^{j+k}_{n}
-
{p}^{j}_{n}
\right)
{\,{\rm d}x}
\nonumber
\\
&
\leq
C_1
h
 \sum_{i=j+1}^{j+k}
 \int_{\Omega}
 |\nabla {p}_{n}^{i}|
 |\nabla\left({p}^{j+k}_{n}-{p}^{j}_{n}
\right)|
{\,{\rm d}x}
\nonumber
\\
&
\quad
+
k h
|\alpha_1|C_f
\int_{\Omega}
\left|
{p}^{j+k}_{n}
-
{p}^{j}_{n}
\right|
{\,{\rm d}x}.
\end{align}
Using \eqref{lipschitz_phi} and \eqref{approximate_problem_04a},
the above estimate can be further rewritten as
\begin{align}
\label{est:703}
&
\quad
\int_{\Omega}
\phi({x},{r}^{j+k}_{n})
\left[
{S}({p}^{j+k}_{n})
-
{S}({p}^{j}_{n})
\right]
\left(
{p}^{j+k}_{n}
-
{p}^{j}_{n}
\right)
{\,{\rm d}x}
\nonumber
\\
&
\leq
C_1
h
 \sum_{i=j+1}^{j+k}
 \int_{\Omega}
 |\nabla {p}_{n}^{i}|
 |\nabla\left({p}^{j+k}_{n}-{p}^{j}_{n}
\right)|
{\,{\rm d}x}
\nonumber
\\
&
\quad
+
\left(
k h
|\alpha_1|C_f
+
khS_s C_{\phi} C_f
\right)
\int_{\Omega}
\left|
{p}^{j+k}_{n}
-
{p}^{j}_{n}
\right|
{\,{\rm d}x}.
\end{align}
Again, we sum \eqref{est:703} for $j=1,\dots,n-k$, multiply it by $h$ and
use \eqref{bound_phi} and \eqref{est:energy_pressure}
to arrive at
\begin{equation}
\label{est:704}
h
 \sum_{j=1}^{n-k}
\int_{\Omega}
\left[
{S}({p}^{j+k}_{n})
-
{S}({p}^{j}_{n})
\right]
\left(
{p}^{j+k}_{n}
-
{p}^{j}_{n}
\right)
{\,{\rm d}x}
\leq
C
k
h,
\qquad
0\leq k < n.
\end{equation}
Further,
using the same arguments as
in \eqref{est:701}--\eqref{est:704},      we arrive at
\begin{equation}\label{est:705}
h \sum_{j=1}^{n-k}
\int_{\Omega}
\left|
{c}^{j+k}_{n}
-
{c}^{j}_{n}
\right|^2
{\,{\rm d}x}
\leq
C
k
h.
\end{equation}
Finally,
derivation similar to that presented above
leads to
\begin{equation}\label{est:706}
h \sum_{j=1}^{n-k}
\int_{\Omega}
\left|
{\vartheta}^{j+k}_{n}
-
{\vartheta}^{j}_{n}
\right|^2
{\,{\rm d}x}
\leq
C
k
h.
\end{equation}
At the same time,
it is easily deduced from \eqref{eq:mamory_mod_01} that
\begin{equation}\label{est:707}
h \sum_{j=1}^{n-k}
\int_{\Omega}
\left|
{r}^{j+k}_{n}
-
{r}^{j}_{n}
\right|^2
{\,{\rm d}x}
\leq
C
k
h.
\end{equation}

\subsection{Temporal interpolants and uniform estimates}
By means of the sequences ${p}^{i}_{n},{c}^{i}_{n},\vartheta^{i}_{n},{r}^{i}_{n}$ constructed
in Section~\ref{sec:approximations}, we define
the piecewise constant interpolants
$\bar{\varphi}_{n}(t) = \varphi^{i}_{n}$ for $t \in ((i - 1){h}, i{h}]$
and, in addition, we extend $\bar{\varphi}_{n}$ for $t\leq 0$ by
$
\bar{\varphi}_{n}(t) = \varphi_{0}$ for $t \in (-{h}, 0]$.
Here, $\varphi^{i}_{n}$ stands for ${p}^{i}_{n},{c}^{i}_{n},\vartheta^{i}_{n}$ or ${r}^{i}_{n}$.

For a function $\varphi$ we often use the simplified notation
$\varphi := \varphi(t)$, $\varphi_{h}(t) := \varphi(t-{h})$,
$\partial_t^{-{h}}\varphi(t) := \frac{\varphi(t) - \varphi(t-{h})}{h}$,
$\partial_t^{h}\varphi(t) := \frac{\varphi(t+{h}) - \varphi(t)}{h}$.
Then,
following \eqref{approximate_problem_01}--\eqref{approximate_problem_03},
the piecewise constant time interpolants
$\bar{p}_{n} \in L^{\infty}(I;W_{\Gamma_D}^{1,s}(\Omega))$,
$\bar{c}_{n} \in L^{\infty}(I;W_{\Gamma_D}^{1,s}(\Omega))$
and
$\bar{\vartheta}_{n} \in L^{\infty}(I;W_{\Gamma_D}^{1,s}(\Omega))$
(with some $s>2$)
satisfy the equations
\begin{align}
\label{eq:1001}
&
\int_{\Omega}
\partial_t^{-{h}} [\phi({x},\bar{r}_{n}(t)){S}(\bar{p}_{n}(t))] \zeta
{\,{\rm d}x}
\nonumber
\\
&
+
\int_{\Omega}
a(x,\bar{p}_{n}(t),\bar{\vartheta}_{n}(t-{h}),\bar{r}_{n}(t-{h}))
\nabla \bar{p}_{n}(t)
\cdot\nabla \zeta
{\,{\rm d}x}
\nonumber
\\
=
&
\int_{\Omega}
\alpha_1 f({x},\bar{p}_{n}(t),\bar{c}_{n}(t-{h}),\bar{\vartheta}_{n}(t-{h})),\bar{r}_{n}(t-{h})))
\zeta
{\,{\rm d}x}
\end{align}
for any $\zeta \in {W_{\Gamma_D}^{1,2}(\Omega)}$,
\begin{align}\label{eq:1002}
&
\quad
\int_{\Omega}
\partial_t^{-{h}}  [ \phi({x},\bar{r}_{n}(t)){S}(\bar{p}_{n}(t))\bar{c}_{n}(t) ] \eta
{\,{\rm d}x}
\nonumber
\\
&
\quad
+
\int_{\Omega}
\phi({x},\bar{r}_{n}(t)){S}(\bar{p}_{n}(t))D_w(\bar{p}_{n}(t-{h})) \nabla \bar{c}_{n}(t) \cdot \nabla\eta
{\,{\rm d}x}
\nonumber
\\
&
\quad
+
\int_{\Omega}
\bar{c}_{n}(t)
a(x,\bar{p}_{n}(t),\bar{\vartheta}_{n}(t-{h}),\bar{r}_{n}(t-{h}))
\nabla \bar{p}_{n}(t)
\cdot \nabla\eta
{\,{\rm d}x}
\nonumber
\\
&
=
0
\end{align}
for any $\eta \in {W_{\Gamma_D}^{1,2}(\Omega)}$
and

\begin{align}\label{eq:1003}
&
\int_{\Omega}
\partial_t^{-{h}} \left[ \phi({x},\bar{r}_{n}(t)){S}(\bar{p}_{n}(t))\bar{\vartheta}_{n}(t)
+
  {\varrho}({x},\bar{r}_{n}(t)) \bar{\vartheta}_{n}(t)
\right] \psi
{\,{\rm d}x}
\nonumber
\\
&
+
\int_{\Omega}
\lambda(x,\bar{p}_{n}(t-{h}),\bar{\vartheta}_{n}(t-{h}),\bar{r}_{n}(t-{h})) \nabla \bar{\vartheta}_{n}(t) \cdot \nabla\psi
{\,{\rm d}x}
\nonumber
\\
&
+
\int_{\Omega}
\bar{\vartheta}_{n}(t)
a(x,\bar{p}_{n}(t),\bar{\vartheta}_{n}(t-{h}),\bar{r}_{n}(t-{h}))
\nabla \bar{p}_{n}(t)
\cdot \nabla\psi
{\,{\rm d}x}
\nonumber
\\
=
&
\int_{\Omega}
\alpha_2 f({x},\bar{p}_{n}(t),\bar{c}_{n}(t-{h}),\bar{\vartheta}_{n}(t-{h})),\bar{r}_{n}(t-{h})))
\zeta
{\,{\rm d}x}
\end{align}
for any $\psi \in {W_{\Gamma_D}^{1,2}(\Omega)}$.
Finally, from \eqref{approximate_problem_04a} and \eqref{approximate_problem_04b} we have
\begin{equation}\label{eq:1004a}
\bar{R}_{n}(t)
=
f({x},\bar{p}_{n}(t),\bar{c}_{n}(t-{h}),\bar{\vartheta}_{n}(t-{h})),\bar{r}_{n}(t-{h})))
\end{equation}
for all $t \in [0,T]$, where
\begin{equation*}
\bar{R}_{n}(t)
=
\frac{r_{n}^{i} - r_{n}^{i-1}}{h}
\quad
\textmd{ for }  t \in ((i - 1){h}, i{h}],
\quad
i=1,2,\dots,n
\end{equation*}
and
$\bar{R}_{n}(0) = r_{n}^{1}/h$, $r_{n}^{0}=0$.
To be able to say something about the behaviour of the sequences
$\left\{ \bar{p}_{n} \right\}$,
$\left\{ \bar{c}_{n} \right\}$,
$\left\{ \bar{\vartheta}_{n} \right\}$,
$\left\{ \bar{R}_{n} \right\}$,
and
$\left\{ \bar{r}_{n} \right\}$,
we now present some apriori estimates for solutions of
the problem \eqref{eq:1001}--\eqref{eq:1004a}.

To this aim, from
\eqref{est:unform_bound_w_05},
\eqref{est:unform_bound_theta_05},
\eqref{est:unform_bound_r_01},
\eqref{est:energy_pressure}, \eqref{est:energy_concentration},
\eqref{est:energy_temperature} and \eqref{est:energy_memory_function}
we see immediately that
\begin{align}
\sup_{0 \leq t \leq T}
\int_{\Omega}
\Theta_{S}(\bar{p}_{n}(t))
{{\rm d}x}
+
\int_0^T \|\bar{p}_{n}(t)\|^2_{W^{1,2}_{\Gamma_D}(\Omega)} {\rm d}t
&\leq C,
\label{est:801}
\\
\int_0^T \|\bar{c}_{n}(t)\|^2_{W^{1,2}_{\Gamma_D}(\Omega)} {\rm d}t
&\leq C,
\label{est:802}
\\
\int_0^T \|\bar{\vartheta}_{n}(t)\|^2_{W^{1,2}_{\Gamma_D}(\Omega)} {\rm d}t
&\leq C,
\label{est:803}
\\
\int_0^T \|\bar{r}_{n}(t)\|^2_{W^{1,2}_{\Gamma_D}(\Omega)} {\rm d}t
&\leq C,
\label{est:803b}
\\
\|\bar{c}_{n}\|_{L^{\infty}({Q_T})}
&\leq C,
\label{est:804}
\\
\|\bar{\vartheta}_{n}\|_{L^{\infty}({Q_T})}
&\leq C,
\label{est:805}
\\
\|\bar{r}_{n}\|_{L^{\infty}({Q_T})}
&\leq C.
\label{est:805}
\end{align}
Moreover,
the estimates \eqref{est:704}--\eqref{est:707}
can be rewritten in the form
\begin{align}
&
\int_0^{T-k{h}}
\left[
{S}(\bar{p}_{n}(t+k{h}))
-
{S}(\bar{p}_{n}(t))
\right]
\left(
\bar{p}_{n}(t+k{h})
-
\bar{p}_{n}(t)
\right)
{\rm d}t
\leq C k {h},
\label{est:806}
\\
&
\int_0^{T-k{h}}
|\bar{c}_{n}(t+k{h}) - \bar{c}_{n}(t) |^2
{\rm d}t
\leq C k {h},
\label{est:807}
\\
&
\int_0^{T-k{h}}
|\bar{\vartheta}_{n}(t+k{h}) - \bar{\vartheta}_{n}(t)|^2
{\rm d}t
\leq {C} k {h},
\label{est:808}
\\
&
\int_0^{T-k{h}}
|\bar{r}_{n}(t+k{h}) - \bar{r}_{n}(t)|^2
{\rm d}t
\leq {C} k {h}.
\label{est:809}
\end{align}

Now we are ready to complete the proof of the main result of this paper
which is the conclusion of the following section.


\subsection{Passage to the limit}
\label{subsec:limit}

The a-priori estimates
\eqref{est:801}--\eqref{est:805}
allow us to conclude that there exist
$p \in L^2(I;W^{1,2}_{\Gamma_D}(\Omega))$,
$c \in L^2(I;W^{1,2}_{\Gamma_D}(\Omega)) \cap L^{\infty}({Q_T})$,
$\vartheta \in L^2(I;W^{1,2}_{\Gamma_D}(\Omega)) \cap L^{\infty}({Q_T})$
and
${r} \in L^{\infty}({Q_T})$,
such that,
letting $n \rightarrow +\infty$ (along a selected subsequence),
\begin{align*}
\bar{p}_{n} & \rightharpoonup  {p}
&&
\textrm{weakly in } L^2(I;W^{1,2}_{\Gamma_D}(\Omega)),
&&
\\
\bar{c}_{n} & \rightharpoonup  c
&&
\textrm{weakly in } L^2(I;W^{1,2}_{\Gamma_D}(\Omega)),
\\
\bar{c}_{n} & \rightharpoonup  c
&&
\textrm{weakly star in } L^{\infty}({Q_T}),
\\
\bar{\vartheta}_{n} & \rightharpoonup  \vartheta
&&
\textrm{weakly in } L^2(I;W^{1,2}_{\Gamma_D}(\Omega)),
\\
\bar{\vartheta}_{n}  & \rightharpoonup  \vartheta
&&
\textrm{weakly star in } L^{\infty}({Q_T}),
\\
\bar{r}_{n} & \rightharpoonup  {r}
&&
\textrm{weakly in } L^2(I;W^{1,2}(\Omega)),
\\
\bar{r}_{n}  & \rightharpoonup  {r}
&&
\textrm{weakly star in } L^{\infty}({Q_T}).
\end{align*}
Thus, we derived fundamental properties of the functions $p$, $c$, $\vartheta$ and $r$.
The crucial step to ensure that $p$, $c$, $\vartheta$ and $r$ solve the problem
\eqref{weak_form_01}--\eqref{eq:memory_weak_form}
consists in showing that the sequences $\left\{ \bar{p}_{n} \right\}$,
$\left\{ \bar{c}_{n} \right\}$,
$\left\{ \bar{\vartheta}_{n} \right\}$
and
$\left\{ \bar{r}_{n} \right\}$,
converge not only weakly in appropriate Bochner spaces, but even
almost everywhere on $Q_T$.

To this aim,
in view of \eqref{est:801} and \eqref{est:806},
using the
compactness argument one can show in the same way
as in \cite[Lemma~1.9]{AltLuckhaus1983} and \cite[Eqs. (2.10)--(2.12)]{FiloKacur1995}
that
\begin{equation}\label{eq555}
S(\bar{p}_{n})   \rightarrow   S(p) \quad \textmd{ in } L^1(Q_T)
\end{equation}
and almost everywhere on $Q_T$.
Since $S$ is strictly monotone,
it follows from \eqref{eq555}
that \cite[Proposition 3.35]{Kacur1990a}
\begin{equation}
\label{conv:p00}
\bar{p}_{n}  \rightarrow  p
\qquad
\textrm{ almost everywhere on } Q_T.
\end{equation}
By similar arguments,
using estimates
\eqref{est:802}--\eqref{est:805} and \eqref{est:807}--\eqref{est:809},
we have
\begin{align}
\bar{c}_{n} & \rightarrow c
&&
\textrm{almost everywhere on } Q_T,
\label{ae_conv_c}
\\
\bar{\vartheta}_{n} & \rightarrow \vartheta
&&
\textrm{almost everywhere on } Q_T,
\label{ae_conv_vartheta}
\\
\bar{r}_{n} & \rightarrow {r}
&&
\textrm{almost everywhere on } Q_T.
\label{ae_conv_r}
\end{align}
Finally,
in consequence of \eqref{eq:mamory_mod_04},
the norms
$\| \bar{R}_{n}(t) \|_{L^{\infty}(\Omega)}$
are uniformly bounded with respect to $n$ and $t$.
Hence, there exists $R \in L^{2}(I,L^{\infty}(\Omega))$,
such that
\begin{equation}\label{weak_conv_R}
\bar{R}_{n}  \rightharpoonup  R
\qquad
\textrm{weakly star in } L^2(I;L^{\infty}(\Omega))
\end{equation}
and $R$ can be shown to satisfy (see \cite[Chapter~11]{Rektorys})
\begin{equation}
\int_0^t R(s){\rm d}s
=
r(t)
\end{equation}
and
\begin{equation}
R(t)
=
r'(t)
\textrm{ in } L^2(I;L^{\infty}(\Omega)).
\end{equation}
It follows that
\begin{equation*}
r \in C([0,T];L^{\infty}(\Omega))
\quad
(\textrm{and even } AC([0,T];L^{\infty}(\Omega)))
\end{equation*}
and $r(0)=0$.
In view of \eqref{conv:p00}--\eqref{ae_conv_r}
and the assumption (iv) we have
\begin{equation*}
f({x},\bar{p}_{n}(t),\bar{c}_{n}(t-{h}),\bar{\vartheta}_{n}(t-{h})),\bar{r}_{n}(t-{h})))
\rightarrow
f({x},{p}(t),{c}(t),{\vartheta}(t)),{r}(t)))
\end{equation*}
almost everywhere on  $Q_T$ and
on account of \eqref{eq:1004a} and \eqref{weak_conv_R}
we can write
\begin{equation*}
R = f
\textrm{ in } L^2(I;L^{\infty}(\Omega)).
\end{equation*}
This leads to \eqref{eq:memory_weak_form}. Moreover,
the above established convergences
are sufficient for taking the limit $n \rightarrow \infty$
in \eqref{eq:1001}--\eqref{eq:1003}
(along a selected subsequence) to get the weak solution
of the system \eqref{strong:eq1a}--\eqref{strong:eq1l}
in the sense of Definition~\ref{def_weak_solution}.
This completes the proof of the main result stated by Theorem~\ref{main_result}.


\subsection*{Acknowledgment}
The first author of this work has been supported by the project GA\v{C}R~16-20008S.
The second author of this work has been supported by the
Croatian Science Foundation (scientific project 3955: Mathematical modeling and numerical simulations of processes in thin or porous domains).


\begin{thebibliography}{99}

\bibitem{AdamsFournier1992}
A.~Adams, J.F.~Fournier,
\emph{Sobolev spaces},
Pure and Applied Mathematics 140,
Academic Press, 2003.

\bibitem{AltLuckhaus1983}
H.W.~Alt, S.~Luckhaus,
\emph{Quasilinear elliptic-parabolic differential equations},
Mathematische Zeitschrift,
{183} (1983) 311--341.



\bibitem{Bear}
J.~Bear,
\emph{Dynamics of Fluids in Porous Media},
Courier Corporation, 1972.




\bibitem{BenesKrupicka2016}
M.~Bene\v{s}, L.~Krupi\v{c}ka,
\emph{Weak solutions of coupled dual porosity flows in fractured rock mass and structured porous media},
Journal of Mathematical Analysis and Applications,
433 (2016) 543--565.



\bibitem{BenesPazanin2017}
M.~Bene\v{s}, I.~Pa\v{z}anin,
\emph{On existence, regularity and uniqueness of thermally coupled
incompressible flows in a system of three dimensional pipes},
Nonlinear Analysis,
149 (2017) 56--80.

\bibitem{BenesPazanin2017b}
M.~Bene\v{s}, I.~Pa\v{z}anin,
\emph{Homogenization of degenerate coupled fluid flows and heat
transport through porous media},
Journal of Mathematical Analysis and Applications,
446 (2017) 165--192.


\bibitem{BenesZeman}
{M.~Bene\v{s}, J.~Zeman},
\emph{Some properties of strong solutions to nonlinear heat and moisture transport in
multi-layer porous structures},
Nonlinear Analysis: Real World Applications,
{13} (2012) 1562--1580.


\bibitem{degond}
P.~Degond, S.~G\'{e}nieys, A.~J\"{u}ngel,
\emph{A system of parabolic equations in nonequilibrium thermodynamics
including thermal and electrical effects},
Journal de Math\'{e}matiques Pures et Appliqu\'{e}es,
76 (1997) 991--1015.


\bibitem{Doktor1985}
A.~Doktor,
\emph{On the solution of the heat equation with nonlinear unbounded memory},
Applications of Mathematics,
30 (1985) 461--474.


\bibitem{FiloKacur1995}
J.~Filo, J.~Ka\v{c}ur,
\emph{Local existence of general nonlinear parabolic systems},
Nonlinear Analysis,
{24} (1995)  1597--1618.





\bibitem{Filo1987}
J.~Filo,
\emph{On solutions of a perturbed fast diffusion equation},
Aplikace matematiky,
32 (1987)  364--380.








\bibitem{gallouet}
T.~Gallou\"{e}t, A.~Monier,
\emph{On the regularity of solutions to elliptic equations},
Rendiconti di Matematica,
19 (1999)  471--488.



\bibitem{Gawin2006a}
D.~Gawin, F.~Pesavento and B.A.~Schrefler,
\emph{Hygro-thermo-chemo-mechanical modelling of concrete at early
ages and beyond. Part I: Hydration and
hygro-thermal phenomena},
International Journal For Numerical Methods In Engineering,
67 (2006) 299--331.

\bibitem{Gawin2011a}
D.~Gawin, F.~Pesavento, B.~Schrefler,
\emph{What physical
phenomena can be neglected when modelling concrete at high temperature? {A}
comparative study. {P}art 1: Physical phenomena and mathematical model},
International Journal of Solids and Structures,
48 (13) 1927--1944.

\bibitem{Gawin2011b}
D.~Gawin, F.~Pesavento, B.~Schrefler,
\emph{What physical
phenomena can be neglected when modelling concrete at high temperature? {A}
comparative study. {P}art 2: Comparison between models.}
International Journal of Solids and Structures,
48 ~(13)  1945--1961.


\bibitem{Genuchten1980}
M.T.~van Genuchten,
\emph{A closed form equation for predicting the hydraulic
conductivity of unsaturated soil},
Soil Science Society of America Journal,
{44} (1980) 892--898.


\bibitem{gerkegenuchten}
H.~Gerke, M.~Van~Genuchten,
\emph{A dual-porosity model for simulating the
preferential movement of water and solutes in structured porous media},
Water Resources Research,
29 (1993) 305--319.


\bibitem{gerkegenuchten_2}
H.~Gerke, M.~Van~Genuchten,
\emph{Evaluation of the first order transfer term for
variably saturated dual porosity flow models},
Water Resources Research,
29 (1993) 1225--1238.






\bibitem{GiaquintaModica1987}
M.~Giaquinta, G.~Modica;,
\emph{Local existence for quasilinear
parabolic systems under nonlinear boundary conditions},
Annali di Matematica Pura ed Applicata,
{149} (1987)  41--59.


\bibitem{groger1989}
K.~Gr\"{o}ger,
\emph{A $W^{1,p}$-estimate for solutions to mixed boundary value
problems for second order elliptic differential equations},
Mathematische Annalen,
283 (1989) 679--687.



\bibitem{harris2017}
P.A.~Harris, E.N.M.~Cirillo, A.~Muntean,
\emph{Weak solutions to Allen-Cahn-like equations modelling consolidation of porous media},
IMA Journal of Applied Mathematics
82 (2017) 224--250.





\bibitem{jungel2000}
A.~J\"{u}ngel,
\emph{Regularity and uniqueness of solutions to a parabolic system in nonequilibrium thermodynamics},
Nonlinear Analysis,
41 (2000)  669--688.


\bibitem{Kacur1990a}
J.~Ka\v{c}ur,
\emph{On a solution of degenerate elliptic-parabolic systems in Orlicz-Sobolev spaces. I},
Mathematische Zeitschrift,
203 (1990)  153--171.



\bibitem{Kacur1999}
J.~Ka\v{c}ur,
\emph{Solution to strongly nonlinear parabolic problems by a linear approximation scheme},
IMA Journal of Numerical Analysis,
19 (1999)  119--145.


\bibitem{Kacur2001}
J.~Ka\v{c}ur,
\emph{Solution of Degenerate Convection-Diffusion Problems by the Method of Characteristics},
SIAM Journal on Numerical Analysis,
39 (2001) 858--879.



\bibitem{KufFucJoh1977}
A.~Kufner, O.~John, S.~Fu\v{c}\'{i}k,
\emph{Function Spaces},
Academia, 1977.



\bibitem{Ladyzhenskaya1968}
O. A.~Ladyzhenskaya, N.N.~Ural'tseva,
Linear and Quasilinelr Equations of Elliptic Type,
Academic Press, New York, 1968.



\bibitem{LiSun2010}
B.~Li, W.~Sun,
\emph{Global existence of weak solution for nonisothermal multicomponent flow
in porous textile media},
SIAM Journal on Mathematical Analysis,
42 (2010) 3076--3102.

\bibitem{LiSunWang2010}
B.~Li, W.~Sun, Y.~Wang,
\emph{Global existence of weak solution to
the heat and moisture transport system in fibrous porous media},
Journal of Differential Equations,
249 (2010)  2618--2642.


\bibitem{LiSun2012}
B.~Li, W.~Sun,
\emph{Global weak solution for a heat and sweat transport
system in three-dimensional fibrous porous media with condensation/evaporation
and absorption},
SIAM Journal on Mathematical Analysis,
{44} (2012)  1448--1473.









\bibitem{IshidaMaekawaKishi2007}
T.~Ishida, K.~Maekawa, T.~Kishi,
\emph{Enhanced modeling of moisture equilibrium and transport in cementitious
materials under arbitrary temperature and relative humidity history},
Cement and Concrete Research,
37 (2007) 565--578.


\bibitem{Maekawa3}
K.~Maekawa, T.~Ishida, T.~Kishi,
\emph{Multi-scale modeling of concrete performance},
Journal of Advanced Concrete Technology,
1 (2003) 91--126.


\bibitem{Maekawa4}
K.~Maekawa, R.~Chaube, T.~Kishi,
\emph{Modelling of concrete performance : hydration, microstructure formation, and mass transport}.
London ; New York : E \& FN Spon, 1999.



\bibitem{necas1967}
J.~Ne\v{c}as,
\emph{Les methodes directes en theorie des equations elliptiques},
Academia, Prague 1967.


\bibitem{Ning1990c}
S.~Ning,
\emph{Mathematical problems on the fluid-solute-heat flow through porous media,
I. Unsaturated case},
Acta Mathematicae Applicatae Sinica,
6 (1990) 135--144.


\bibitem{Ning1990}
S.~Ning,
\emph{Mathematical problems on the fluid-solute-heat flow through porous media,
II. Partially saturated case},
Acta Mathematicae Applicatae Sinica,
6 (1990) 145--157.


\bibitem{Ning1990b}
S.~Ning,
\emph{An elliptic-parabolic system arising from the fluid-solute-heat flow
through saturated porous media},
Acta Mathematicae Applicatae Sinica,
6 (1990) 224--237.



\bibitem{Ning1993}
S.~Ning,
\emph{Multidimensional degenerate diffusion problem with evolutionary boundary condition:
existence, uniqueness, and approximation},
In: Flow in Porous Media,
Volume 114 of the series ISNM International Series of Numerical Mathematics,
(1993) 165--178.



\bibitem{Ozbolt}
J.~O\v{z}bolt, G.~Balabani\'{c}, G.~Peri\v{s}ki\'{c}, M.~Ku\v{s}ter,
\emph{Modelling the effect of damage on transport processes in concrete},
Construction and Building Materials,
24 (2010)  1638--48.


\bibitem{PinderGray}
G.F.~Pinder, W.G.~Gray,
\emph{Essentials of Multiphase Flow in Porous Media},
Wiley-Interscience, New Jersey, 2008.


\bibitem{Pluschke1988}
V.~Pluschke;
\emph{Local solution of parabolic equations with strongly increasing nonlinearity
by the Rothe method},
Czechoslovak Mathematical Journal,
38 (1988) 642--654.

\bibitem{Pluschke1992}
V.~Pluschke,
\emph{Rothe's Method for Semilinear Parabolic Problems with Degeneration},
Mathematische Nachrichten,
156 (1992) 283--295.


\bibitem{Pluschke1997}
V.~Pluschke,
\emph{Rothe's Method for Degenerate Quasilinear Parabolic Equations},
In: Zuzana Dosla and J. Kuben and Jaromir Vosmansky (eds.): Proceedings of Equadiff 9,
Conference on Differential Equations and Their Applications, Brno, August 25-29, 1997,
[Part~3] Papers. Masaryk University, Brno, 1998. CD-ROM,
pp. 247--254.



\bibitem{Rektorys}
K.~Rektorys,
\emph{The method of discretization in time and partial differential equations},
Reidel Co, Dodrecht, Holland 1982.


\bibitem{Roubicek2005}
T.~Roub\'{i}\v{c}ek,
\emph{Nonlinear Partial Differential Equations With Applications.}
Birkh\"{a}user, 2005.


\bibitem{Song2201}
H.-W.~Song, H.-J.~Cho, S.-S.~Park, K.-J.~Byun, K.~Maekawa,
\emph{Early-age cracking resistance evaluation of concrete structures},
Concrete Science and Engineering,
 3  (2001)  63--72.


\bibitem{Vala2002}
J.~Vala,
\emph{On a system of equations of evolution with a
non-symmetrical parabolic part occuring in the analysis of moisture
and heat transfer in porous media},
Applications of Mathematics,
{47} (2002) 187--214.


\bibitem{Weidemaier1991}
P.~Weidemaier,
\emph{Local existence for parabolic problems with
fully nonlinear boundary condition; An $L_p$-approach},
Annali di Matematica Pura ed Applicata,
{160} (1991)  207--222.


\bibitem{Wu2006}
Z.~Wu, J.~Yin, Ch.~Wang,
\emph{Elliptic and Parabolic Equations},
World Scientific, 2006.

\end{thebibliography}
\end{document}